%% file: main.tex
\documentclass[12pt]{amsart}

\usepackage[paperheight=11in, paperwidth=8.5in, left=1in, top=1in, right=1in, bottom=1in]{geometry}
\usepackage[linktocpage=true, linktoc=all, colorlinks=true, linkcolor=black, citecolor=black, filecolor=black, urlcolor=black, pagebackref=false, pdfstartpage={1}, pdfstartview={FitH}, pdftitle={}, pdfauthor={}, pdfsubject={}, pdfcreator={}, pdfproducer={}, pdfkeywords={}]{hyperref}
\usepackage{graphicx} 
\usepackage{amsthm,amsfonts, amssymb, amscd}
\usepackage{cancel, comment, float}

\usepackage{enumerate}
\usepackage{youngtab}
\usepackage{ytableau}
\usepackage[all]{xy}
\usepackage{euscript}
\usepackage{cite}
\usepackage{tikz}
\usepackage{multicol}
\usepackage{subcaption}
\usepackage{stmaryrd}
\usetikzlibrary{matrix, backgrounds}

\newenvironment{bsmallmatrix}
  {\left[\begin{smallmatrix}}
  {\end{smallmatrix}\right]}

\newtheorem{theorem}{Theorem}[section]
\newtheorem{lemma}[theorem]{Lemma}
\newtheorem{corollary}[theorem]{Corollary}
\newtheorem{prop}[theorem]{Proposition}

\theoremstyle{remark}
\newtheorem{remark}[theorem]{Remark}
\theoremstyle{remark}
\newtheorem{example}[theorem]{Example}
\theoremstyle{definition}

\theoremstyle{definition}

\theoremstyle{definition}
\newtheorem{eg_no_qed}[theorem]{Example}
\newenvironment{ex}[1][]{\begin{eg_no_qed}[#1]\pushQED{\qed}}{\popQED\end{eg_no_qed}}
\AfterEndEnvironment{ex}{\noindent\ignorespaces}
\newtheorem{definition}[theorem]{Definition}

\newtheorem*{claim}{Claim}

\numberwithin{equation}{section}

\theoremstyle{plain}
\newtheorem{innercustomgeneric}{\customgenericname} 
\providecommand{\customgenericname}{}

\makeatletter 
\newcommand{\newcustomtheorem}[2]{%
  \newenvironment{#1}[1]
  {%
   \renewcommand\customgenericname{#2}%
   \def\theinnercustomgeneric{##1}%
   \@ifnextchar[{\innercustomgeneric}{\innercustomgeneric}%
  }
  {\endinnercustomgeneric}
}
\makeatother

\newcustomtheorem{customthm}{Theorem}
\newcustomtheorem{customconj}{Conjecture}
\newcustomtheorem{customprop}{Proposition}
\newcustomtheorem{customlem}{Lemma}

\newcommand{\C}{\mathbb{C}}           
\newcommand{\PP}{\mathbb{ P}} 

\renewcommand{\ker}{\operatorname{ker }}

\newcommand{\GL}{\operatorname{GL}}

\newcommand{\cH}{\mathcal{H}}

\tikzset{partition/.style={fill,circle,inner sep=1pt}}
\usetikzlibrary{positioning}
\usetikzlibrary{decorations.pathreplacing}
\usetikzlibrary{matrix,arrows}
\usetikzlibrary{calc}
\tikzset{partition/.style={fill,circle,inner sep=1pt},
         part/.style={baseline=0,scale=0.5,bend left=45},
         partlabel/.style={below}}
\usetikzlibrary{shapes,arrows}
\usetikzlibrary{snakes}
\tikzstyle{pnt}=[draw,ellipse,fill,inner sep=1pt]
\tikzstyle{opnt}=[draw,ellipse,inner sep=1pt]
\tikzstyle{opnt}=[ ]
\tikzstyle{pntt}=[draw,ellipse,fill,inner sep=0.5pt]
\tikzstyle{point}=[draw,ellipse,fill,inner sep=2pt]

\newcommand\cristina[1]{{\color{teal} \sf $\triangle$ Cristina: [#1]}}

\newcommand{\hleft}{\langle}
\newcommand{\hright}{\rangle}

\title{Recursive formula for the Equations of Hessenberg varieties}

\author{Hikari Iwasaki}
\address{Department of Mathematics\\ Stanford University\\ 450 Jane Stanford Way\\ Stanford, California\\ 94305\\ USA  }
\email{iwasakih@stanford.edu}

\author{Cristina Sabando-Alvarez}
\address{Department of Mathematics\\ Washington University in St. Louis \\ One Brookings Drive\\ St. Louis, Missouri\\ 63130\\ USA }
\email{sabandoalvarez@wustl.edu}

\begin{document}

\begin{abstract}
   Hessenberg varieties are subvarieties of the flag variety, defined by containment conditions on flags with respect to a linear operator. The study of these varieties lies in the intersection of algebraic geometry, combinatorics, and representation theory. In this paper, we develop an algebro-geometric procedure for determining the closed subvariety structure of a Hessenberg variety $\mathcal{H}(X,h)$ in the flag variety for any linear operator $X$ and Hessenberg function $h$, by imposing a partial order on the Hessenberg functions and analyzing the relation of the corresponding Hessenberg varieties. In particular, we give a concrete recursive formula for determining all equations cutting out a given Hessenberg variety in each Schubert cell. As an application, we provide an alternative geometric proof of Tymoczko's results on the existence of affine pavings of a given Hessenberg variety and on the dimension count of its cells.
\end{abstract}

\maketitle

\section{Introduction}
\label{sec: intro}

\input{intro}

\section{Preliminary Notion: The Flag Variety}
\label{sec: flag}
\input{flag_variety/flag_variety}

\section{Preliminary Notion: Hessenberg Varieties}
\label{sec: hessenberg}
\input{hessenberg/hessenberg_variety}

\section{Main Theorem: A Recursive Formula}
\label{sec: thm 1}

\input{Thm1/thm1}

\section{Prior result: Tymoczko's affine paving}

\input{Julianna_connection/Tymoczko_intro}

\section{Rederivation of Tymoczko's results from a geometric perspective using Theorem \ref{thm: Thm 1}}
\label{sec: rederivation}
As an application of Theorem \ref{thm: Thm 1}, we rederive Tymoczko's results on paving by affines of Hessenberg varieties, stated in Theorem \ref{thm: Cor 6.3} and Theorem \ref{thm: Tymoczko's dw}, from a geometric perspective. We first address the affineness condition of nonempty strata in Section~\ref{sec: affineness of strata}  and their dimension count in Section \ref{sec:dimstrata}. Finally, we address the nonempty condition of a stratum in Section \ref{sec: thm 2}.

\subsection{Affineness of the nonempty intersection $C_w \cap  \cH(X,h)$}
\input{Julianna_connection/highest_form}

\input{Thm2/thm2}


\bibliographystyle{alpha}
\bibliography{refs}  

\end{document}

%% file: intro.tex
Hessenberg varieties are central to geometric representation theory and combinatorics as their cohomology rings carry group actions that link the geometry of flag varieties with the theory of symmetric functions.
In this paper we study intersections of Hessenberg varieties with Schubert cells and develop a procedure to describe their structure as a closed subvariety of the flag variety. Using the poset structure on Hessenberg functions we obtain a recursive formula to  determine these intersections.

The \emph{flag variety} $\operatorname{Fl}(V)$ is the moduli space of flags $V_\bullet = (V_1\subseteq V_2\subseteq \cdots \subseteq V_n=V)$, which are nested collections of subspaces $V_i$ of dimension $i$ in an $n$-dimensional vector space $V$ over a  field $k$.  When $V = k^{n}$, the flag variety can alternatively be described as the quotient space $\operatorname{GL}_n/B$ of the general linear group by its Borel subgroup $B$ of upper-triangular matrices.


Hessenberg varieties, originally defined by De Mari, Procesi, and Shayman \cite{DeMariProcesiShayman1992} are subvarieties of the flag variety which are parametrized by a linear operator $X: k^n\rightarrow k^n$ and a non-decreasing function $h:\{1,2,\ldots,n\}\rightarrow\{1,2,\ldots,n\}$ satisfying $h(j)\geq j$, called a \emph{Hessenberg function}. More precisely, the \emph{Hessenberg variety} $\cH(X,h)$ is defined as the closed subvariety of the flag variety consisting of flags $V_\bullet$ such that $XV_i\subseteq V_{h(i)}$ for all $i$.
 



The flag variety admits a decomposition into $B$-orbits $C_w = BwB/B$, called \emph{Schubert cells}, which are parametrized by elements $w$ of the symmetric group $S_n$. Moreover, these Schubert cells are affine spaces, so they form an affine stratification (a paving by affines) of the flag variety, compatible with the Bruhat order on $S_n$.

As with the flag variety, combinatorial and geometric methods are central to studying the structure of Hessenberg varieties. However, while the full flag variety admits a straightforward paving by affine Schubert cells, intersecting a Hessenberg variety $\mathcal{H}(X,h)$ with these cells yields a much more intricate partition of the variety. Tymoczko~\cite{Tymoczko2006} showed that all Hessenberg varieties in type A admit a paving by affines, obtained by intersecting them with the Bruhat decomposition of the flag variety. This generalizes earlier results of De Concini--Procesi~\cite{deConciniProcesi} on related varieties. The result was later extended by Precup~\cite{Precup2013} to a broader class of Hessenberg varieties beyond Type A.

The existence of such an affine paving bridges the geometric properties and the underlying combinatorics of Hessenberg varieties, reducing the computation of topological invariants such as Betti numbers, Hodge numbers, and the topological Euler characteristic to a systematic counting of the dimensions of the affine cells.


\bigskip

The main result of this paper is an explicit algebro-geometric description of the closed subscheme structure of the Hessenberg variety $\cH(X,h)$, for any linear operator $X$ and any Hessenberg function $h$, when restricted to a Schubert cell.
We use the following perspective: rather than regarding the Hessenberg function $h$ to be a fixed input, we impose a partial order relation $\leq$ on the space of Hessenberg functions, and describe the closed subscheme relations of the corresponding Hessenberg varieties. More precisely, we can define partial ordering on Hessenberg functions so that two Hessenberg functions $h$ and $ h'$ satisfy $h' \leq h$ if and only if $h'(i) \leq h(i)$ for all $i$. Then  the associated Hessenberg varieties satisfy a closed embedding relation $\cH(X,h')\subset \cH(X,h)$ as closed subschemes of the flag variety. 
In the Main Theorem, we show that if $h$ and $h'$ differ only at one index by value 1, then the closed embedding relation $\cH(X,h')\subset \cH(X,h)$ is given by the vanishing locus of a single equation in Schubert variables of $C_w$.
\begin{customthm}{1}
[Main Theorem, informal statement]
\label{thm: intro thm 1}
    Let $X$ be any linear operator on $V$.
Suppose that $h $ and $ h'$ are two Hessenberg functions that differ only at index $k$ with value $h(k) = h'(k)+1$.
If we describe flags in $C_w$ in terms the Schubert canonical matrix form, then $C_w \cap \mathcal{H}(X, h')$ is cut out in $ C_w \cap \mathcal{H}(X, h)$ by a single equation in Schubert variables, which is determined by a recursive formula.
\end{customthm}
\noindent The precise statement of the Main Theorem, together with an explicit form of the recursive formula, is given in Theorem \ref{thm: Thm 1}.


Using the Main Theorem, we can determine the equations cutting out the Hessenberg variety $\cH(X,h)$ in each Schubert cell $C_w$ for any linear operator $X$ and  any Hessenberg function $h$ as follows.  Given a Hessenberg function $h$, we construct a Hessenberg path to $h$ from the \emph{top Hessenberg function} $h^{\operatorname{top}}$, defined as $h^{\operatorname{top}}(i) = n$ for all $i$, with respect to the partial order defined above.  Then given a fixed Schubert cell $C_w$, for each edge along the Hessenberg path, we apply the Main Theorem to obtain an equation cutting out the Hessenberg variety in $C_w$. The collection $S_{MT}$ of such equations along the Hessenberg path cuts out the Hessenberg variety in the Schubert cell $C_w$. Moreover, the collection $S_{MT}$ does not depend on the choice of the Hessenberg path from $h$ to $h^{\operatorname{top}}$.
 


\bigskip

As an application of this result, we present alternative rederivations of \cite[Corollary 6.3]{Tymoczko2006} by Tymoczko, deduced via an algebro-geometric perspective.
This result provides a condition for the intersection $C_w\cap \cH(X,h)$ to be nonempty, and proves each nonempty intersection is an affine space and computes and its dimension. To prove this result, Tymoczko introduced the notion of a matrix in \emph{highest form}. In our proof, we use a stronger notion called a matrix in \emph{strictly highest form},
which can be characterized as follows. A matrix is in strictly highest form if its columns are either zero or a standard basis vector, and it has the property that if $\phi(k)$ denotes the row-index of the pivot of the $k$-th column vector, then $\phi(k)$ must be non-decreasing in $k$ and strictly increasing once positive, with the convention that $\phi(k) = 0$ if the $k$-th column vector is zero.


Using our Main Theorem, along with the partial ordering $\leq$ of Hessenberg functions and $\phi$, we rederive results by Tymoczko from a geometric perspective as follows.
First, for affineness and the dimension counting of the nonempty intersection $C_w \cap \cH(X,h)$, we provide the following reduction.
\begin{customlem}{1}\label{intro:lem1}
  Let $X$ be a nilpotent matrix and assume that $C_w \cap \mathcal{H}(X,h)$ is nonempty. Let $S$ be a finite set of nonzero equations cutting $C_w \cap \mathcal{H}(X,h)$ in $C_w$. To show $C_w \cap \mathcal{H}(X,h)$ is an affine space, it suffices to show that 
\begin{enumerate}
    \item for each $f_\alpha \in S$, we can associate a Schubert variable $a^{(\alpha )}$ such that $f_{\alpha}$ can be decomposed as 
    $f_{\alpha} = -a^{(\alpha)} + g_{\alpha},$
    where $g_{\alpha}$ is independent of $a^{(\alpha)}$, and
    \item there exists a strict total order $\prec$ of equations in $S$ such that the variable $a^{(\alpha)}$ does not appear in any of the equations succeeding $f_{\alpha}$ in $S$.
\end{enumerate}
Moreover, if  (1) and (2) hold, then the codimension of $C_w \cap \cH(X,h)$ in $C_w$ is the cardinality  of $S$.  
\end{customlem}

In Section \ref{sec: rederivation}, we show that if $X$ is in strictly highest form and $S$ is the collection $S_{MT}$ of nonzero equations $f_{(k,m)}$ obtained by recursively applying the Main Theorem to each edge of a fixed Hessenberg path from $h$ to $h^{\operatorname{top}}$, then Statements (1) and (2) indeed hold; the proofs are given in Proposition \ref{prop: f(k,m) has a linear term} and Proposition \ref{prop: existence of total order}, respectively.  
\noindent The statement of Lemma \ref{intro:lem1} is also introduced as Lemma \ref{lem: reduction of affineness to total order} in Section 6.

\bigskip

Next, the condition for the intersection $C_w \cap \cH(X,h)$ to be nonempty can be written as a simple inequality on the Hessenberg function as follows.
\begin{customthm}{2}
\label{thm: Thm 2, intro}
Suppose that $X$ is a nilpotent matrix in strictly highest form and $w \in S_n$. There exists a unique Hessenberg function $h_{\operatorname{min}}$ such that $C_w \cap \mathcal{H}(X, h)$ is nonempty if  and only if $h \geq h_{\operatorname{min}}$, where $h_{\operatorname{min}}$ is given by
\begin{align*}
    h_{\min}(k) = \begin{cases}
         \max\{k, h_{\operatorname{min}}(k-1)\} &\text{if  } w(k) \leq \dim(\ker(X)), \\
        \max\{k, h_{\min}(k-1), w^{-1}(\phi(w(k)))\}  &\text{if  } w(k) > \dim(\ker(X)),
    \end{cases}
\end{align*}
\noindent where we use the convention $h_{\operatorname{min}}(0) = 0$.
\end{customthm}
\noindent Theorem \ref{thm: Thm 2, intro} is also stated as Theorem \ref{thm: Thm 2} in Section 6. As a corollary, this theorem presents a formula for the unique minimal Hessenberg function $h_{\text{min}}$ whose associated Hessenberg variety intersects a given Schubert cell.

\bigskip


This paper is structured as follows. In Section \ref{sec: flag}, we give definitions and key properties of the flag variety, as well as set up notation for the description of Schubert cells. In Section \ref{sec: hessenberg}, we introduce Hessenberg varieties with key examples. In Section \ref{sec: thm 1}, we provide the formal statement and proof of Theorem \ref{thm: intro thm 1}, which provides a recursive formula for determining the equations involved in the closed subscheme structure of a Hessenberg variety in each Schubert cell. In Section 5, we describe Tymoczko's results given in \cite[Corollary 6.3]{Tymoczko2006} as Theorems \ref{thm: Cor 6.3} and \ref{thm: Tymoczko's dw}. The precise definition of the highest form of a nilpotent operator and its properties are presented in this section. In Section 6, as application of the Main Theorem, we rederive Theorems \ref{thm: Cor 6.3} and \ref{thm: Tymoczko's dw} from algebro-geometric perspective.
Section 6.1 is dedicated to the proof of affineness of nonempty intersection and its dimension, including Lemma \ref{intro:lem1}.
In Section 6.2, we rederive the nonemptiness condition of the intersection $C_w \cap \cH(X,h)$ of the Hessenberg variety with a Schubert cell, proving Theorem \ref{thm: Thm 2, intro}.

\subsection*{Acknowledgments} 
This work was initiated at the Collaborative Workshop in Algebraic Geometry, held at the Institute for Advanced Study in 2024, and we thank this institute for its warm hospitality.
We thank Julianna Tymoczko for her topic suggestion and group guidance during the workshop. We are grateful to our advisors Martha Precup and Ravi Vakil for helpful discussions. CSA was partially supported by NSF CAREER grant DMS-2237057.

%% file: flag_variety/flag_variety.tex
In this paper, we assume that $n$ is a positive integer, $V$ is an $n$-dimensional vector space over an algebraically closed field $k$, and $e_1,\cdots, e_n$ form the standard basis of $V$. For background on the flag variety, we refer the reader to \cite{billey2025,Fulton1997} and state only the results relevant to the scope of this paper.


A \emph{flag} $V_\bullet$ in a vector space $V$ is a chain of subspaces,
$$0 = V_0 \subset V_1 \subset \cdots \subset V_n = V,$$
where each $V_i$ has dimension $i$. 
    The \emph{flag variety} $\operatorname{Fl}(V)$ is the moduli space of all flags on $V$. If $V \cong k^{n}$, we also denote by $\operatorname{Fl}_n$.
Flags can be described in matrix form, which leads to the notion of Schubert cells as follows. Given a flag $V_\bullet$, we assign to it an invertible $n\times n$-matrix whose first $i$ column vectors span the subspace $V_i$.   We use the following convention:
the matrix has exactly one \emph{pivot} entry equal to 1 in each row and column, with 0's below and to the right of the pivots. This construction uniquely determines the matrix associated with the given flag, which will be called the \emph{(Schubert) canonical matrix} associated to the flag $V_\bullet$. The \emph{Schubert cell associated with the permutation $w \in S_n$} is the set $C_w$ of
all flags with pivots at position $w$, in the sense that the associated canonical matrix sends $e_i$ to $e_{w(i)}$ for all $i$. By construction, the Schubert cells are affine spaces.
For the rest of the paper, we denote a permutation $w \in S_n$ by $[w(1)\cdots w(n)]$. For example, if $w$ is the transposition of $1$ and $2$, then we denote $w = [2134\cdots n]$.
\begin{ex}
The Schubert cells of $\operatorname{Fl}_3$ are given by 
    \begin{align*}
        C_{[321]} &= \Big\{\begin{bsmallmatrix}
        a_{11} & a_{12} & 1 \\
        a_{21} & 1 & 0 \\
        1 & 0 & 0
    \end{bsmallmatrix}  \ \Big| \ a_{11}, a_{12}, a_{21} \in \C \Big\}, \\
        C_{[231]} &= \Big\{\begin{bsmallmatrix}
        a_{11} & a_{12} & 1 \\
        1 & 0 & 0 \\
        0 & 1 & 0
    \end{bsmallmatrix}  \ \Big| \ a_{11}, a_{12} \in \C \Big\}, 
        &C_{[312]} &= \Big\{\begin{bsmallmatrix}
        a_{11} & 1 & 0 \\
        a_{21} & 0 & 1 \\
        1 & 0 & 0
    \end{bsmallmatrix}  \ \Big| \ a_{11},  a_{21} \in \C \Big\}, \\
        C_{[213]} &= \Big\{\begin{bsmallmatrix}
        a_{11} & 1 & 0 \\
        1 & 0 & 0 \\
        0 & 0 & 1
    \end{bsmallmatrix}  \ \Big| \ a_{11} \in \C \Big\},
        & C_{[132]} &= \Big\{\begin{bsmallmatrix}
        1 & 0 & 0 \\
        0 & a_{22} & 1 \\
        0 & 1 & 0
    \end{bsmallmatrix}  \ \Big| \ a_{22} \in \C \Big\}, \\
        C_{[123]} &= \Big\{\begin{bsmallmatrix}
        1 & 0 & 0 \\
        0 & 1 & 0 \\
        0 & 0 & 1
    \end{bsmallmatrix} \Big\}.
    \end{align*}
\end{ex}

\noindent Note that the subspace $V_i$ of the flag $V_\bullet$ spanned by the first $i$ column vectors of an invertible matrix $M$ is invariant under column operations which involve the first $i$ column vectors. 
It follows that for every invertible upper triangular matrix $b$, the matrix $Mb$ represents the same flag. Therefore, the flag variety has the following group-theoretic description: 
    \begin{equation}\label{flagident}
        \operatorname{Fl}_n \cong \GL_n/B. 
    \end{equation} 
Under the identification of flags with matrix cosets as above, the Schubert cell $C_w$ is the collection $BwB/B$ of cosets in $\operatorname{GL}_n/B$. 

Recall that a \emph{stratification} of a topological space $X$ is the data of a partition $X = \coprod_{i \in I} X_i$ into locally closed subsets $X_i$ and  a partial ordering on $I$, such that for each $j \in I$, $\overline{X}_j = \coprod_{i \leq j} X_i$. The locally closed subsets $X_i$ are called the \emph{strata} of the stratification.  When all of the strata are the form $\mathbb{A}^n$ for some $n$, the stratification is also a \emph{paving by affines}. The flag variety has a stratification by Schubert cells $C_w$ with respect to the Bruhat partial order on $w \in S_n$ as follows. The \emph{Bruhat order} of the symmetric group $S_n$ is a partial order $\leq$ such that $v \leq w$ if and only if for every $p, q \in \{1,\cdots, n\}$, the following inequality holds:
$$\#\{i \in \{1,\cdots, p\}  \ | \ v(i) \leq q \} \geq \#\{i \in \{1,\cdots, p\} \ | \ w(i) \leq q \}.$$

\begin{prop}[Stratification of the flag variety]
    \label{prop: stratification by Schubert cells}
    The flag variety has a stratification by the Schubert cells under Bruhat order, which is a paving by affines.
\end{prop}

\begin{ex}
To illustrate how the closure of a Schubert cell contains another Schubert cell,
    we return to the example of the flag variety $\operatorname{Fl}_3$ over $\C$. 
    Consider the Schubert cell $$C_{[321]} = \Big\{\begin{bsmallmatrix}
        a_{11} & a_{12} & 1 \\
        a_{21} & 1 & 0 \\
        1 & 0 & 0
    \end{bsmallmatrix}  \ \Big| \ a_{11}, a_{12}, a_{21} \in \C \Big \}. $$ 
    If $a_{12}$ is nonzero, the matrix corresponding to an element of this Schubert cell is
    \begin{align*}
        \begin{bmatrix}
        a_{11} & a_{12} & 1 \\
        a_{21} & 1 & 0 \\
        1 & 0 & 0
    \end{bmatrix}  \sim \begin{bmatrix}
        a_{11} & 1 & 1 \\
        a_{21} & 1/a_{12} & 0 \\
        1 & 0 & 0
    \end{bmatrix} \sim \begin{bmatrix}
        a_{11} & 1 & 0 \\
        a_{21} & 1/a_{12} & -1/a_{12} \\
        1 & 0 & 0
    \end{bmatrix}   \sim \begin{bmatrix}
        a_{11} & 1 & 0 \\
        a_{21} & 1/a_{12} & 1 \\
        1 & 0 & 0
    \end{bmatrix}  ,
    \end{align*}
    where the notation $\sim$ indicates that the matrices describe the same flag, or equivalently, they represent the same coset in the group-theoretic description of the flag variety, $\operatorname{Fl}_n \cong \GL_n/B$.    
    Now, as $1/a_{12}$ approaches $0$, the second column vector approaches the vector $\begin{bsmallmatrix}
        1 \\ 0 \\ 0
    \end{bsmallmatrix}$. Hence the limit of this matrix is an element of the following Schubert cell, 
    $$C_{[312]}= \Big\{\begin{bsmallmatrix}
        a_{11} & 1 & 0 \\
        a_{21} & 0 & 1 \\
        1 & 0 & 0
    \end{bsmallmatrix}  \ \Big| \ a_{11},  a_{21} \in \C \Big\},$$
    and conversely, any element of $C_{[312]}$ is realized as a limit of elements in $C_{[321]}$.
    This agrees with the fact that the Bruhat order on $S_n$ satisfies $[312] \leq [321]$ and Proposition \ref{prop: stratification by Schubert cells}.
\end{ex}
As an example, Figure \ref{fig: Bruhat order S3 and S4} shows the Bruhat order of the symmetric group $S_3$ and the closure relations of Schubert cells of the flag variety $\operatorname{Fl}_3$; the arrows in the figure on the right indicate the closure relations, where the target of the arrow indicates a largest Schubert cell in the boundary of the source Schubert cell.

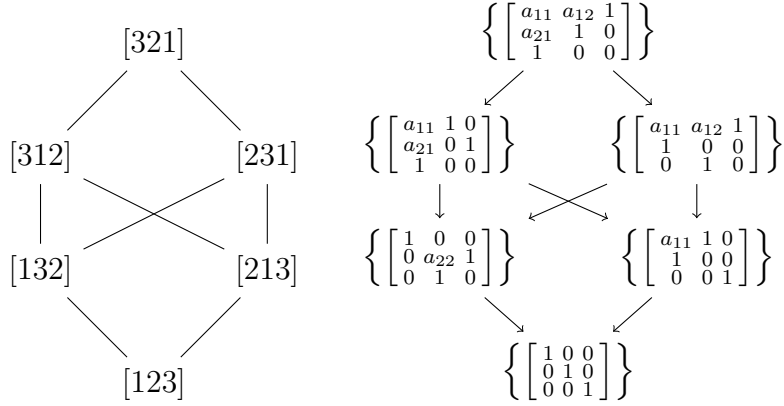
\begin{figure}[ht!]
    \centering
    \begin{tikzpicture}[scale=0.75]
    \node (321) at (0,6) {$[321]$};
    \node (231) at (2,4) {$[231]$};
    \node (312) at (-2,4) {$[312]$};
    \node (213) at (2,2) {$[213]$};
    \node (132) at (-2,2) {$[132]$};
    \node (123) at (0,0) {$[123]$};
    \draw[-] (321) -- (231);
    \draw[-] (321) -- (312);
    \draw[-] (312) -- (132);
    \draw[-] (312) -- (213);
    \draw[-] (231) -- (132);
    \draw[-] (231) -- (213);
    \draw[-] (213) -- (123);
    \draw[-] (132) -- (123);
    \end{tikzpicture}
    \quad
    \begin{tikzpicture}[xscale=0.85, yscale=0.75]
    \node (321) at (0,6) {$\left\{ \begin{bsmallmatrix}
        a_{11} & a_{12} & 1 \\
        a_{21} & 1 & 0 \\
        1 & 0 & 0
    \end{bsmallmatrix} \right\}$};
    \node (231) at (2,4) { $\left\{ \begin{bsmallmatrix}
        a_{11} & a_{12} & 1 \\
        1 & 0 & 0 \\
        0 & 1 & 0
    \end{bsmallmatrix} \right\}$ };
    \node (312) at (-2,4) { $\left\{ \begin{bsmallmatrix}
        a_{11} & 1 & 0 \\
        a_{21} & 0 & 1 \\
        1 & 0 & 0
    \end{bsmallmatrix} \right\}$};
    \node (213) at (2,2) { $\left\{ \begin{bsmallmatrix}
        a_{11} & 1 & 0 \\
        1 & 0 & 0 \\
        0 & 0 & 1
    \end{bsmallmatrix} \right\}$};
    \node (132) at (-2,2) { $\left\{ \begin{bsmallmatrix}
        1 & 0 & 0 \\
        0 & a_{22} & 1 \\
        0 & 1 & 0
    \end{bsmallmatrix} \right\}$};
    \node (123) at (0,0) { $\left\{ \begin{bsmallmatrix}
        1 & 0 & 0 \\
        0 & 1 & 0 \\
        0 & 0 & 1
    \end{bsmallmatrix} \right\}$};
    
    \draw[->] (321) -- (231); 
    \draw[->] (321) -- (312); 
    \draw[->] (312) -- (213);   
    \draw[->] (312) -- (132); 
    \draw[->] (231) -- (213);  
    \draw[->] (231) -- (132);   
    \draw[->] (213) -- (123);  
    \draw[->] (132) -- (123);  
    \end{tikzpicture}

\caption{The Hasse diagrams for Bruhat order on $S_3$ (left) and the corresponding closure relations of the Schubert cells  of $\operatorname{Fl}_3$ (right)}
\label{fig: Bruhat order S3 and S4}
\end{figure}

%% file: hessenberg/hessenberg_variety.tex
A Hessenberg variety is a subvariety of the flag variety $\operatorname{Fl}(V)$ parametrized by two objects: a linear operator $X$ on $V$ and a Hessenberg function $h$.  For background on Hessenberg varieties, we refer the reader to \cite{AbeHoriguchi2019}.

\begin{definition}
    A \emph{Hessenberg function on $n$} is a function $h : \{1,\cdots, n\} \to \{1,\cdots, n\}$
    satisfying 
    \begin{enumerate}
        \item $h(i) \geq i$ for each $i\in\{1,\ldots,n\}$, and
        \item $h(i+1) \geq h(i)$ for all $i\in\{1,\ldots,n-1\}$.
    \end{enumerate}
\end{definition}
\noindent For the rest of this paper, we denote a Hessenberg function $h$ on $n$ in the form $\hleft h(1)\cdots h(n) \hright$. For example, the Hessenberg function given by the identity function for all $i$ will be denoted $\hleft 12\cdots n \hright$, and the Hessenberg function given by  $h(i) = n$ for all $i$, also called the \emph{top Hessenberg function}, will be denoted $h^{\operatorname{top}} =\hleft n\cdots n \hright$.
We can now define Hessenberg varieties.
\begin{definition}
Let $X$ be a linear operator on $V$ and $h$ a Hessenberg function on $n$. The associated \emph{Hessenberg variety} is 
\begin{align*}
    \mathcal{H}(X,h) := \{V_\bullet \in \operatorname{Fl}_n \ | \ XV_i \subset V_{h(i)} \text{ for all }i~\}.
\end{align*}
\end{definition}
\noindent For example, $\mathcal{H}(0,h)$ is the flag variety for every Hessenberg function $h$, and  $\mathcal{H}(X,h^{\operatorname{top}})$ is the flag variety for every linear operator $X$. 

Under the identification of $\operatorname{Fl}_n$ with $\operatorname{GL}_n / B$ and flag as $B$-cosets, we get the following alternative description of the Hessenberg variety:
$$
\mathcal{H}(X, h)=\left\{g B \in \operatorname{GL}_n/ B \mid g^{-1} X g \in H(h)\right\},
$$
where $H(h)$ is the \emph{Hessenberg space} associated to $h$, defined as $H(h) := \{M \in M_{n\times n} \ | \ M_{ij} = 0 \text{ if } i > h(j) \}$ where $M_{n\times n}$ denotes all $n\times n$ matrices over $k$.



 \begin{remark}
 \label{rmk: H(X,h) depend only on conjugacy class}
     Given a Hessenberg function, there is an isomorphism of algebraic varieties $\mathcal{H}(X, h) \simeq \mathcal{H}\left(g^{-1} X g, h\right)$ for all $g \in G$. In particular, Hessenberg varieties depend only on the conjugacy classes of the linear operator. When $X$ is a nilpotent operator, this implies that Hessenberg varieties depend only on the Jordan type of $X$.
 \end{remark}

A standard approach for studying Hessenberg varieties is to utilize the stratification of the flag variety by Schubert cells, as presented in Proposition~\ref{prop: stratification by Schubert cells}. By analyzing the intersection of the Hessenberg variety with each Schubert cell $C_w$, the problem reduces to examining local linear algebraic conditions on the affine coordinates $a_{ij}$ of the canonical matrix representing the cell. We refer to these coordinates as \emph{Schubert variables}.

\begin{ex}
Consider the case where the nilpotent matrix $X_\lambda$ is subregular, i.e., when $\lambda = (n-1, 1)$, and the Hessenberg function is the identity map $\hleft 1\cdots n \hright$. Then the Hessenberg variety $\mathcal{H}(X_{(n-1,1)}, \hleft 1\cdots n \hright)$ is a chain of $n-1$ copies of the projective line $\PP^1$.
For example, in the case $n = 3$, the flags $V_\bullet$ in the Hessenberg variety $\mathcal{H}(X_{(2,1)}, \hleft 123 \hright)$ are given by
    \begin{align*}
        \begin{bmatrix} 
        a & 1 & 0 \\
        0 & 0 & 1 \\
        1 & 0 & 0
        \end{bmatrix} \text{ where } a \in k,  \ 
        \begin{bmatrix} 
        1 & 0 & 0 \\
        0 & d & 1 \\
        0 & 1 & 0
        \end{bmatrix}, \text{ where }d \in k,  \ \text{ and }
        \begin{bmatrix} 
        1 & 0 & 0 \\
        0 & 1 & 0 \\
        0 & 0 & 1
        \end{bmatrix}.
    \end{align*}
    To see that this is a chain of two copies of $\PP^1$, we observe that
    as $a^{-1}$ approaches $0$, the Schubert cell meets $d=0$ of the second Schubert cell. As $d^{-1} $ approaches $0$, the Schubert cell meets the last Schubert cell (the identity). This is shown in Figure \ref{fig: subregular example}.

\begin{figure}[ht!]
    \centering
    \begin{tikzpicture}
      \fill[blue] (-4,0) circle (2pt) ;
      \fill[red] (0,0) circle (2pt) ;
      \fill[violet] (4,0) circle (2pt);

      \node[blue] at (-3.5,0.3) {$a=0$};
      \node[blue] at (-0.75,0.3) {$a=\infty$};
      \node[red] at (0.5,0.3) {$d = 0$};
      \node[red] at (3.5,0.3) {$d=\infty$};
      
      \node[blue] at (-2,-1.5) {$\left\{\begin{bsmallmatrix}
          a & 1 & 0 \\ 0 & 0 & 1 \\ 1 & 0 & 0
      \end{bsmallmatrix}\right\}$};
      \node[red] at (2,-1.5) {$\left\{\begin{bsmallmatrix}
          1 & 0 & 0 \\ 0 & d & 1 \\ 0 & 1 & 0
      \end{bsmallmatrix}\right\}$};
      \node[violet] at (4,-1.5) {$\left\{\begin{bsmallmatrix}
          1 & 0 & 0 \\ 0 & 1 & 0 \\ 0 & 0 & 1
      \end{bsmallmatrix}\right\}$};
      
      \draw[red] (0,0) .. controls (1,-1.) and (3,-1.) .. (4,0);
      \draw[blue] (-4,0) .. controls (-3,-1) and (-1,-1) .. (0,0);
    \end{tikzpicture}
    \caption{The Hessenberg variety $\mathcal{H}(X_{(2,1)}, [123])$ as a chain of two $\PP^1$.}
    \label{fig: subregular example}
\end{figure}
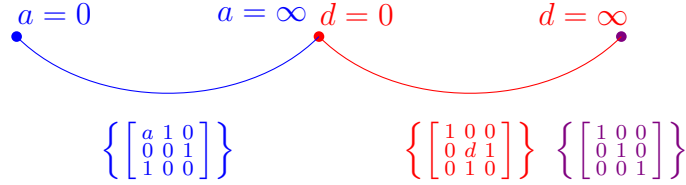
\end{ex}

We conclude this section with a description of a partial order of the Hessenberg functions, which will be a key tool for the main results of this paper presented in the next sections.

\begin{definition}
    \label{prop: poset structure on hessenberg functions}
 Let $h$ and $h'$ be Hessenberg functions on $n$. We say $h' \leq h$ if $h'(i) \leq h (i)$ for all $i$.
\end{definition} 

\noindent This poset structure on Hessenberg functions translates to closed embedding relations on the corresponding Hessenberg varieties as follows: given two Hessenberg functions $h'$ and $h$ with relation $h' \leq h$, we have an inclusion $\mathcal{H}(X,h') \subset \mathcal{H}(X,h)$ of Hessenberg varieties, which is a closed embedding. To understand this embedding, we define a partial order to track Hessenberg varieties whose Hessenberg functions differ by one at a certain index.
\begin{definition}\label{poseth}
Let $h$ and $h'$ be Hessenberg functions on $n$. We say that $h$ \emph{covers $h'$ at index $k$ with value $m$} if
    \[h(k)=h'(k)+1=m~\text{and}~ h_2(i)=h_1(i)~\text{for all}~i \neq k.\]
\end{definition}

\noindent For example, $\hleft 333 \hright$ covers $\hleft 233\hright$ at index $1$ with value $3$. 
As a non-example of a partial order relation, observe that $h_1=\hleft 1344 \hright $ and $h_2=\hleft 2244 \hright$ have no partial order relation between them. In this case,  we can identify that the (unique) smallest Hessenberg function larger than $h_1$ and $h_2$ is $\hleft 2344\hright$, and the (unique) largest Hessenberg function smaller than both $h_1$ and $h_2$ is $\hleft 1244\hright$;  see Figure \ref{fig: poset structure for n=3 and n=4}.

\begin{figure}[ht!]
    \centering
    \begin{tikzpicture}[scale=0.75]
    \node (333) at (0,6)  {$\hleft 333  \hright$};
    \node (233) at (0,4.5)  {$\hleft 233  \hright$};
    \node (223) at (1.5,3)  {$\hleft 223  \hright$};
    \node (133) at (-1.5,3)  {$\hleft 133  \hright$};
    \node (123) at (0,1.5)  {$\hleft 123  \hright$};
    
    \draw[->] (333) --(233);
    \draw[->] (233) --(223);
    \draw[->] (233) --(133);
    \draw[->] (223) --(123);
    \draw[->] (133) --(123);
    \end{tikzpicture}
    \qquad
    \begin{tikzpicture}[scale=0.75]
    \node (4444) at (0,9)  {$\hleft 4444  \hright$};
    \node (3444) at (0,7.5)  {$\hleft 3444  \hright$};
    \node (2444) at (-2,6)  {$\hleft 2444  \hright$};
    \node (3344) at (2,6)  {$\hleft 3344  \hright$};
    \node (1444) at (-2,4.5)  {$\hleft 1444  \hright$};
    \node (2344) at (0,4.5)  {$\hleft 2344  \hright$};
    \node (3334) at (2,4.5)  {$\hleft 3334  \hright$};
    \node (2244) at (0,3)  {$\hleft 2244  \hright$};
    \node (1344) at (-2,3)  {$\hleft 1344  \hright$};
    \node (2334) at (2,3)  {$\hleft 2334  \hright$};
    \node (1244) at (-2,1.5)  {$\hleft 1244  \hright$};
    \node(2234) at (2,1.5)  {$\hleft 2234  \hright$};
    \node(1334) at (0,1.5)  {$\hleft 1334  \hright$};
    \node(1234) at (0,0)  {$\hleft 1234  \hright$};
    \draw[->] (1334) -- (1234);
    \draw[->] (2234) -- (1234);
    \draw[->] (1244) -- (1234);
    \draw[<-] (1244) -- (2244);
    \draw[<-] (1244) -- (1344);
    \draw[->] (1344) -- (1334);
    \draw[<-] (1334) -- (2334);
    \draw[<-] (2234) -- (2334);
    \draw[<-] (2334) -- (3334);
    \draw[<-] (2234) -- (2244);
    \draw[<-] (1344) -- (1444);
    \draw[<-] (2244) -- (2344);
    \draw[<-] (1344) -- (2344);
    \draw[<-] (2334) -- (2344);
    \draw[<-] (1444) -- (2444);
    \draw[<-] (2344) -- (2444);
    \draw[<-] (2344) -- (3344);
    \draw[<-] (2444) -- (3444);
    \draw[<-] (3334) --(3344);
    \draw[<-] (3344)-- (3444);
    \draw[<-] (3444) --(4444);
    \end{tikzpicture}

    \caption{Hasse diagram of the partial order of Hessenberg functions for $n=3$ (left) and $n=4$ (right)}
    
    \label{fig: poset structure for n=3 and n=4}
\end{figure}
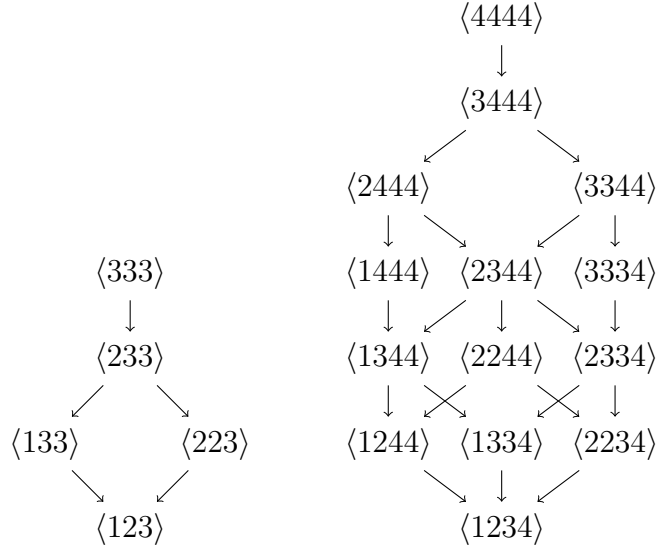


Finally, observe that for each Hessenberg function $h$, there exists a path from $h$ to $h^{\operatorname{top}}$ with respect to the partial order, in the sense that there are relations $h = h_0 \leq h_1 \leq \cdots \leq h_r = h^{\operatorname{top}}$, where $h_{i+1}$ covers $h_i$ for all $i$. Such a path will be called a \emph{Hessenberg path} from $h$ to $h^{\operatorname{top}}$. Note that it is possible to have multiple Hessenberg paths from a fixed $h$ to $h^{\operatorname{top}}$; see Figure \ref{fig: poset structure for n=3 and n=4}.

%% file: Thm1/thm1.tex

As the main result of this paper, we present a concrete recursive formula for identifying the closed subscheme structure of Hessenberg varieties in flag variety using the stratification by Schubert cells. 
The idea is as follows.
Let $\cH(X,h)$ be any Hessenberg variety, and fix a Hessenberg path from $h$ to $h^{\operatorname{top}}$:
$$h = h_0 \leq h_1 \leq h_2 \leq \cdots \leq h_r =h^{\operatorname{top}},$$ where $h_{i+1}$ covers $h_{i}$ for all $i$. The associated Hessenberg varieties also form a chain $$\cH(X,h)  = \cH(X, h_0)\subset \cdots \subset\cH(X,h_{i}) \subset \cH(X,h_{i+1}) \subset \cdots \subset \cH(X, h_r)  = \cH(X,h^{\operatorname{top}})=\operatorname{Fl}(V)$$ of closed embeddings.
In particular, to understand a Hessenberg variety $\cH(X,h)$ as a closed subscheme of the flag variety $\operatorname{Fl}(V)$, it suffices to understand the relation of two Hessenberg varieties $\cH(X,h_{i+1}) \subset \cH(X,h_{i})$, one covering the other, introduced in Section \ref{sec: hessenberg}. Hence in the following, we focus on the case where $h$ covers $h'$  at index $k$ with value $m$, so that we have an embedding of the corresponding Hessenberg varieties, $\mathcal{H}(X, h') \subset \mathcal{H}(X, h)$. 
The Main Theorem completely determines the equation(s) cutting out this embedding when restricted to a Schubert cell.

{
}


\subsection{Proof Strategy}


 
We are interested in the closed subscheme structure of $\cH(X,h')$ in $\cH(X,h)$. Assume that $h$ covers $h'$ at index $k$ with value $m$.

Let $V_\bullet$ be a flag in $\cH(X,h)$.   Denote by $v_i$ the $i$th column vector of $V_\bullet$ in the Schubert canonical matrix form. The following result provides an equivalent condition for this flag to also lie in $\cH(X,h')$. 
\begin{lemma}
    \label{lem: Xvk in Vm-1 condition}
    We have $V_\bullet \in \cH(X,h')$ if and only if $V_\bullet \in \cH(X,h)$ and $Xv_k \in V_{m-1}$.
\end{lemma}
\begin{proof}
    By unwinding the definition of Hessenberg varieties, the claim is equivalent to showing $XV_i \subset V_{h'(i)}$ for all $i$ if and only if $XV_i \subset V_{h(i)}$ for all $i$ and $Xv_k \in V_{m-1}$. Because $h$ and $h'$ have the same value at $i \neq k$, this is equivalent to the following claim:
    \begin{claim}
        Assume $XV_i \subset V_{h'(i)} = V_{h(i)}$ for $i \neq k$. Then $XV_k \subset V_{m-1}$ if and only if $XV_k \subset V_m$ and $Xv_k \in V_{m-1}$.
     \end{claim}
     The forward implication is clear. For the backward implication, because the subspace $V_k$ is spanned by the first $k$ vectors of $V_\bullet$, it suffices to show that $Xv_i \in V_{m-1}$ for all $i \leq k$. The case $i = k$ is taken care of by the assumption $Xv_k \in V_{m-1}$. The case $ i < k$ is implied by the assumption $XV_i \subset V_{h
     (i)} = V_{h'(i)}$ and the inequality $h'(i)\leq h'(k)= m-1$.
\end{proof}

It follows that to describe the closed subscheme structure of $\cH(X,h')$ in $\cH(X,h)$,  we need to understand the condition $X v_k \in V_{m-1}$ in Lemma \ref{lem: Xvk in Vm-1 condition} algebraically.  To do this, denote the column vectors of the Schubert canonical matrix of the flag $V_\bullet$ by $v_1,\cdots, v_n$ as before. By expanding $Xv_k$ uniquely as $Xv_k = c_1v_1+\cdots+c_n v_n$, we observe that the condition $Xv_k \in V_m$ is equivalent to $c_{m+1}=\cdots = c_n = 0$, so $V_\bullet \in \cH(X,h)$ lies in $\cH(X,h')$ if and only if $c_{m} = 0$ in this expansion. 
To give an algebraic formulation of the condition $c_m = 0$, we work in the Schubert cell $C_w$ in which $V_\bullet$ lies, so that the vectors $v_i$ and consequently the coefficients $c_i$ can be expressed in terms of the corresponding Schubert variables $a_{i,j}$. This procedure results in the Main Theorem, which presents a constructive recursive formula for equations cutting out $C_w \cap \cH(X,h')$ in $C_w \cap \cH(X,h)$ in terms of the Schubert variables of the canonical Schubert matrix for $C_w$.

\begin{theorem}[Main Theorem, precise statement]
\label{thm: Thm 1} Let $X$ be any linear operator on $V$.
Suppose that $h $ and $ h'$ are two Hessenberg functions such that $h$ covers $h' $ at index $k$ with value $m$.
Let $V_\bullet$ be a flag in $C_w \cap \mathcal{H}(X, h)$ described 
in the canonical matrix form of $C_w$ in terms of Schubert variables $a_{i,j}$. Then
the closed embedding $C_w \cap \mathcal{H}(X, h') \subset C_w \cap \mathcal{H}(X, h)$ is cut out by a single equation, which is given by the following recursive formula:
\begin{align}
    \label{eq: thm 1 recursive equation}
    f_{(k,m)}  =  \Big[Xv_k - \sum_{w(j)> w(m)}  f_{(k,j)} v_{j}\Big]_{w(m)}. \tag{$\ast$}
\end{align}
\end{theorem}
This is a formal statement of Theorem \ref{thm: intro thm 1} from the Introduction.
\begin{proof} 
    Denote the column vectors of the canonical matrix of the flag by  $v_1,\cdots, v_n$. By Lemma \ref{lem: Xvk in Vm-1 condition}, we wish to express the condition $Xv_k \in V_{m-1}$. By expanding $Xv_k$ uniquely with respect to the basis $\{v_1,\cdots, v_n\}$ as $Xv_k = c_1 v_1 + \cdots + c_n v_n$ and noting that $Xv_k \in V_m$ holds, the condition $Xv_k \in V_{m-1}$ is equivalent to $c_m = 0$. Therefore, the equation $f_{(k,m)}$ we wish to determine is the coefficient $c_m$ in this expansion. 

    We use Gaussian elimination to find the coefficients $c_i$ inductively as follows. Denote $v^{(q)} := v_{w^{-1}(q)}$ and $c^{(q)} := c_{w^{-1}(q)}$. The vector $v^{(q)}$ corresponds to the column vector of the Schubert cell whose pivot is at row index $q$. We expand $Xv_k$ in the basis of $V$ in the order of decreasing pivot positions: $\{v^{(n)},\cdots, v^{(1)}\}$. We determine the coefficients $c^{(i)}$ of the unique expansion
    $$Xv_k = c^{(n)}v^{(n)} + \cdots + c^{(1)}v^{(1)}$$
    using Gaussian elimination as follows.
    As the base case, we set $c^{(n)} = [Xv_k]_{n}$. Then $y_0 := Xv_k - c^{(n)} v^{(n)}$ has a lower pivot row-index than $n$.
    Next, we set $c^{(n-1)}: = [y_0]_{n-1}$. Then $y_1 := Xv_k  - c^{(n)} v^{(n)} - c^{(n-1)} v^{(n-1)}$ has a lower pivot row-index than $n-1$.
    For the general case, we set $c^{(n-(j+1))}: = [y_j]_{j+1}$ and $y_{j+1} = Xv_k - \sum_{l \geq n-j }  c^{(l)} v^{(l)}$. 
    This gives a recursion relation:
    $$c^{(n-j)}  = \Big[Xv_k - \sum_{l > n-j}  c^{(l)} v^{(l)}\Big]_{n-j}. $$
    By substituting $w^{-1}(n-j) = m$, we obtain a polynomial equation for $f_{(k,m)} = c_m$ in terms of Schubert variables of $C_w$, as claimed.
\end{proof}

\begin{corollary}
Let $h$ be a Hessenberg function. Take any path from the top Hessenberg function $h^{\operatorname{top}}$ to $h$ in the Hasse diagram of the partial order on Hessenberg functions. Then the Hessenberg variety $\cH(X,h)$ is cut out in $C_w$ by the collection of equations given by Theorem \ref{thm: Thm 1} at each edge of the path. $\qed$
\end{corollary}

\begin{remark}
    In Theorem \ref{thm: Thm 1}, $X$ is an arbitrary linear operator that can be taken to be any representative of the same conjugacy class if necessary. In Section 6, we utilize this property and take $X$ to be a particular representative of a conjugacy class of a nilpotent matrix.
\end{remark} 

\begin{ex}
  Consider the case where $n = 3$ and the nilpotent operator $X_\lambda$ is the Jordan canonical form of Jordan type $(2,1)$, given by $X_{(2,1)} = \begin{bsmallmatrix}
    0 & 1 & 0 \\ 0 & 0 & 0  \\ 0 & 0 & 0 
\end{bsmallmatrix}$.
Using Theorem \ref{thm: Thm 1}, we determine the equations cutting out $C_w \cap \mathcal{H}(X_\lambda, h)$ in $C_w$ for all $w \in S_n$ for a fixed Hessenberg function $h$ as follows.

\input{Thm1/example_2_1_shorter}

\end{ex}
 \input{Thm1/example_2_1_summary}

\input{Thm1/example_2_1_table}

Using Table \ref{table: (2,1)}, we can deduce some geometric properties of each Hessenberg variety associated with Jordan type $(2,1)$. For example, we get the following observation of  the Hessenberg variety $\mathcal{H}(X_{(2,1)}, \hleft 233  \hright)$.
\begin{prop}
    The Hessenberg variety $\mathcal{H}(X_{(2,1)}, \hleft 233  \hright)$ has two irreducible components, each of dimension 2, whose  intersection is the chain of two copies of $\PP^1$.

\end{prop}

\begin{proof}
    
From Table \ref{table: (2,1)}, we observe that $\mathcal{H}(X_{(2,1)}, \hleft 233  \hright)$ has two irreducible components, each of dimension 2 from $w = [321]$ and $[312]$:
$$\overline{\Bigg\{\begin{bmatrix}
    a_{11} & a_{12} & 1 \\ 
    0 & 1 & 0 \\ 
    1 & 0 & 0
\end{bmatrix}\ \Big| \ a_{11}, a_{12} \in k \Bigg\}} \text{ and } \overline{\Bigg\{\begin{bmatrix}
    a_{11} & 1 & 0 \\ 
    a_{21} & 0 & 1 \\ 
    1 & 0 & 0
\end{bmatrix}\ \Big| \ a_{11}, a_{21} \in k \Bigg\}}, $$
which intersect exactly at two one-dimensional cells 
$$\overline{\Bigg\{\begin{bmatrix}
    a_{11} & 1 & 0 \\ 
    0 & 0 & 1 \\ 
    1 & 0 & 0
\end{bmatrix}\ \Big| \ a_{11}\in k \Bigg\}} \text{ and } \overline{\Bigg\{\begin{bmatrix}
    1 & 0 & 0 \\ 
    0 & a_{22} & 1 \\ 
    0 & 1 & 0
\end{bmatrix}\ \Big| \ a_{22}\in k \Bigg\}},$$
which intersect exactly at the zero-dimensional cell 
$${\Bigg\{\begin{bmatrix}
    1 & 0 & 0 \\ 
    0 & 0 & 1 \\ 
    0 & 1 & 0
\end{bmatrix}\Bigg\}}.$$
\end{proof}

\noindent 





\begin{ex}
Let $n=4$ and $w=[4213]$ which has Schubert cell $$C_{[4213]} = \Bigg\{\begin{bsmallmatrix}
             a_{11} & a_{12} & 1 & 0 \\
             a_{21} & 1 & 0 & 0  \\
             a_{31} & 0 & 0 & 1 \\
             1 & 0 & 0 & 0
         \end{bsmallmatrix}  \ \Big| \ a_{ij} \in k  \Bigg\}, $$
    and the linear operator $X =  \begin{bsmallmatrix}
        0 & 0 & 1 & 0 \\
        0 & 0 & 0 & 1 \\
        0 & 0 & 0 & 0 \\
        0 & 0 & 0 & 0 
    \end{bsmallmatrix}$. Note that $X$ is nilpotent of Jordan type (2,2). From the following matrix multiplication,
    \begin{align*}
   \begin{bmatrix}
        0 & 0 & 1 & 0 \\
        0 & 0 & 0 & 1 \\
        0 & 0 & 0 & 0 \\
        0 & 0 & 0 & 0 
    \end{bmatrix}
    \begin{bmatrix}
             a_{11} & a_{12} & 1 & 0 \\
             a_{21} & 1 & 0 & 0  \\
             a_{31} & 0 & 0 & 1 \\
             1 & 0 & 0 & 0
     \end{bmatrix} = 
    \begin{bmatrix}
         a_{31} & 0 & 0 & 1  \\
        1 & 0 & 0 & 0 \\
         0 & 0 & 0 & 0 \\
        0 & 0 & 0 & 0 
     \end{bmatrix} ,
    \end{align*}
    we observe that the column vectors expand in terms of the basis $\{v^{(4)},\cdots, v^{(1)}\} = \{v_1,\cdots, v_4\} $ as follows.
    \begin{equation*}
        \begin{cases}
        X v_1 = v_2  - (a_{12}-a_{31})v_3 \\
        X v_2 = 0, \\
        X v_3 = 0, \\
        X v_4 = v_3 .
    \end{cases} 
    \end{equation*}
    
    This leads to the following equations and diagram for $w = [4213]$ in Figure \ref{fig: (2,2), w=[4213]}.    
    \begin{align*}       
    \begin{cases}
        f_{1,2} = 1,  \quad  f_{1,3} = a_{31}-a_{12} ,\quad f_{1,4} =0, \\
        f_{2,3} = 0,  \quad  f_{2,4} = 0, \\
        f_{3,4} = 0.
    \end{cases}
    \end{align*}
     
\end{ex}

\begin{figure}[H]
\centering
\begin{tikzpicture}
\node (4444) at (0,9) {$\hleft 4444 \hright$};
\node (3444) at (0,7.5) {$\hleft 3444 \hright$};
\node (2444) at (-2,6) {$\hleft 2444 \hright$};
\node (3344) at (2,6) {$\hleft 3344 \hright$};
\node (1444) at (-2,4.5) {$\hleft 1444 \hright$};
\node (2344) at (0,4.5) {$\hleft 2344 \hright$};
\node (3334) at (2,4.5) {$\hleft 3334 \hright$};
\node (2244) at (0,3) {$\hleft 2244 \hright$};
\node (1344) at (-2,3) {$\hleft 1344 \hright$};
\node (2334) at (2,3) {$\hleft 2334 \hright$};
\node (1244) at (-2,1.5) {$\hleft 1244 \hright$};
\node(2234) at (2,1.5) {$\hleft 2234 \hright$};
\node(1334) at (0,1.5) {$\hleft 1334 \hright$};
\node(1234) at (0,0) {$\hleft 1234 \hright$};

\draw[->] (1334)  [green,double] --node[left, green]  { $0$} (1234);
\draw[->] (2234) [red,dashed] -- node[below right,red]{$1$} node[red]{$\times$} (1234);
\draw[->] (1244) [green,double] -- node[below left,green]{$0$}(1234);
\draw[<-] (1244)[red,dashed] -- node[above,red]{$1$} node[red]{$\times$} (2244);
\draw[<-] (1244) [green,double] --node[left, green]  { $0$} (1344);
\draw[->] (1344)[green,double] -- node[below,green]{$0$}  (1334);
\draw[<-] (1334)[red,dashed] -- node[ above,red]{$1$} node[red]{$\times$} (2334);
\draw[<-] (2234)[green,double] --node[right, green]  { $0$} (2334);
\draw[<-] (2334)[blue] --node[right,blue]  { $a_{31}-a_{12}$}  (3334);
\draw[<-] (2234) [green,double]-- node[below,green]{$0$} (2244);
\draw[<-] (1344) [green, double]--node[ left,green]  { $0$}   (1444);
\draw[<-] (2244)[green,double] --node[left, green]  { $0$} (2344);
\draw[<-] (1344)[red,dashed] -- node[ above,red]{$1$} node[red]{$\times$} (2344);
\draw[<-] (2334)[green,double] -- node[below,green]{$0$}(2344);
\draw[<-] (1444) [red,dashed]--  node[ left,red]{$1$} node[red]{$\times$} (2444);
\draw[<-] (2344) [green, double]--node[below left,green]  { $0$}  (2444);
\draw[<-] (2344)[blue] --node[above left, blue]  { $a_{31}-a_{12}$} (3344);
\draw[<-] (2444) [blue]-- node[above left,blue]  { $a_{31}-a_{12}$} (3444);
\draw[<-] (3334)[green, double] -- node[ right,green]{$0$} (3344);
\draw[<-] (3344)[green, double]-- node[above right,green]{$0$}(3444);
\draw[<-] (3444) [green, double]--node[left, green] {$0$}(4444);
\end{tikzpicture}
\caption{A partially ordered diagram of Hessenberg varieties with cutting equations for the operator $X=\begin{bsmallmatrix}
        0 & 0 & 1 & 0 \\
        0 & 0 & 0 & 1 \\
        0 & 0 & 0 & 0 \\
        0 & 0 & 0 & 0 
    \end{bsmallmatrix}$ and  $w = [4213]$.}
\label{fig: (2,2), w=[4213]}
\end{figure}

As seen in these examples, geometric properties of Hessenberg varieties, such as dimensions and irreducible components, can be retrieved directly by looking at the equations cutting out $C_w \cap \cH(X,h)$ in the Schubert cells $C_w$ and varying $w \in S_n$.

%% file: Thm1/example_2_1_shorter.tex
To start, take $w = [321]$. The canonical matrix form $M$ of the Schubert cell is given by 
     $$M = \begin{bmatrix}
        a_{11} & a_{12} & 1 \\
        a_{21} & 1 & 0 \\
        1 & 0 & 0
    \end{bmatrix} = \begin{bmatrix}
        v_{1} & v_2 & v_3
    \end{bmatrix}.$$
    We compute and expand $X_\lambda v_k$ in terms of the basis $\{v_i\}$:
    \begin{align*}
        X_\lambda M = \begin{bmatrix}
        a_{21} & 1 & 0 \\
        0 & 0 & 0 \\
        0 & 0 & 0
    \end{bmatrix} \implies \begin{cases}
        X_\lambda v_1 = a_{21} v_3, \\
        X_\lambda v_2 = v_3, \\
        X_\lambda v_3 = 0. 
    \end{cases}
    \end{align*}
    Therefore, 
    $f_{1,2} =0, \  f_{1,3} = a_{21}, \text{ and } f_{2,3} = 1 . $
    This yields the following partially ordered diagram of Hessenberg varieties with corresponding equations cutting the succeeding Hessenberg variety given in Figure 
    \ref{fig:lambda = (2,1), w = [321]}. The green thickened arrows indicate that the additional equation $f_{(k,m)}$ is $0$, or equivalently, no additional equation cutting out the closed subscheme, and the red dashed arrows indicate the additional equation  $f_{(k,m)}$ is 1, i.e., the closed subscheme cut by the equation $f_{(k,m)}$ is empty.
    
    \begin{figure}[H]
    \centering
    \begin{tikzpicture}
    \node (333) at (0,6) {$\hleft 333  \hright$};
    \node (233) at (0,4.5) {$\hleft 233  \hright$};
    \node (223) at (1.5,3) {$\hleft 223  \hright$};
    \node (133) at (-1.5,3) {$\hleft 133  \hright$};
    \node (123) at (0,1.5) {$\hleft 123  \hright$};
    
    \draw[->] (333) [blue]--node[left,blue]  { $a_{21} $} (233);
    \draw[->] (233) [red,dashed]--node[ red]{$\times$}node[right, red]{$\quad 1$}  (223);
    \draw[->] (233) [green,double]--node[green, left] {$0 \quad$}  (133);
    \draw[->] (223) [green,double]-- node[green, right] {$\quad  0 \quad$}  (123);
    \draw[->] (133) [red,dashed]--node[red]{$\times$} node[left, red]{$1  \quad$}  (123);
    \end{tikzpicture}
    \caption{Equations cutting the Hessenberg varieties at Schubert cell $C_{[321]}$ associated with the Jordan canonical form of Jordan type $(2,1)$}
    \label{fig:lambda = (2,1), w = [321]}
    \end{figure}

    \noindent To obtain the final equation cutting out $C_w \cap \mathcal{H}(X_\lambda , h)$ in $C_w$ for any $w$, we take any path from the top Hessenberg function $h^{\columnseprulecolor{top}}$ to $h$, and collect equations $f_{(k,m)}$ on the path.
    For example, from this diagram, we can read that $C_{[321] } \cap \mathcal{H}(X_{(2,1)}, \hleft 233  \hright)$ and  $C_{[321] } \cap \mathcal{H}(X_{(2,1)}, \hleft 133  \hright)$ are cut out by the equation $a_{21} = 0$, while $C_{[321] } \cap \mathcal{H}(X_{(2,1)}, \hleft 223  \hright)$ and $C_{[321] } \cap\mathcal{H}(X_{(2,1)}, \hleft 123  \hright)$ are empty.

 By applying the same procedure to all the other Schubert cells, we obtain a collection of diagrams of equations cutting out each Schubert cell for each Hessenberg function, presented in Figure \ref{fig: diagrams for lambda=(2,1)}. These diagrams can be summarized to give Table \ref{table: (2,1)}.

%% file: Thm1/example_2_1_summary.tex
\begin{figure}[H]
\begin{subfigure}{0.3\textwidth}
  \centering
    \begin{tikzpicture}
    \node (333) at (0,6) {$\hleft 333  \hright$};
    \node (233) at (0,4.5) {$\hleft 233  \hright$};
    \node (223) at (1.5,3) {$\hleft 223  \hright$};
    \node (133) at (-1.5,3) {$\hleft 133  \hright$};
    \node (123) at (0,1.5) {$\hleft 123  \hright$};
    
    \draw[->] (333) [blue]--node[left,blue]  { $a_{21} = 0$} (233);
    \draw[->] (233) [red, dashed]--node[red]{$\times$} node[right,red]{$1$}  (223);
    \draw[->] (233) [green, double]-- node[left,green]{$0$} (133);
    \draw[->] (223) [green, double]-- node[right,green]{$0$} (123);
    \draw[->] (133) [red, dashed]--node[red]{$\times$} node[left,red]{$1$} (123);
    \end{tikzpicture}
    
    \caption{$w = [321]$}
\end{subfigure}
\hfill
\begin{subfigure}{0.3\textwidth}
  \centering
    \begin{tikzpicture}
    \node (333) at (0,6) {$\hleft 333  \hright$};
    \node (233) at (0,4.5) {$\hleft 233  \hright$};
    \node (223) at (1.5,3) {$\hleft 223  \hright$};
    \node (133) at (-1.5,3) {$\hleft 133  \hright$};
    \node (123) at (0,1.5) {$\hleft 123  \hright$};
    
    \draw[->] (333) [red, dashed]--node[red]{$\times$} node[right,red]{$1$} (233);
    \draw[->] (233) [green, double]--node[right,green]{$0$} (223);
    \draw[->] (233) [green, double]-- node[left,green]{$0$}(133);
    \draw[->] (223) [green, double]-- node[right,green]{$0$} (123);
    \draw[->] (133) [green, double]-- node[left,green]{$0$} (123);
    \end{tikzpicture}
    \caption{$w = [231]$}
\end{subfigure}
\hfill
\begin{subfigure}{0.3\textwidth}
    \centering
    \begin{tikzpicture}
    \node (333) at (0,6) {$\hleft 333  \hright$};
    \node (233) at (0,4.5) {$\hleft 233  \hright$};
    \node (223) at (1.5,3) {$\hleft 223  \hright$};
    \node (133) at (-1.5,3) {$\hleft 133  \hright$};
    \node (123) at (0,1.5) {$\hleft 123  \hright$};
    
    \draw[->] (333) [green, double]--node[right,green]{$0$} (233);
    \draw[->] (233) [green,double]--node[right,green]{$0$} (223);
    \draw[->] (233) [blue]--node[left,blue]{$a_{21} = 0$} (133);
    \draw[->] (223) [blue]--node[right,blue]{$a_{21}= 0$} (123);
    \draw[->] (133) [green,double]-- node[left,green]{$0$}(123);
    \end{tikzpicture}

    \caption{$w = [312]$}
\end{subfigure}

\vspace{2em}

\begin{subfigure}{0.3\textwidth}
    \centering
    \begin{tikzpicture}
    \node (333) at (0,6) {$\hleft 333  \hright$};
    \node (233) at (0,4.5) {$\hleft 233  \hright$};
    \node (223) at (1.5,3) {$\hleft 223  \hright$};
    \node (133) at (-1.5,3) {$\hleft 133  \hright$};
    \node (123) at (0,1.5) {$\hleft 123  \hright$};
    
    \draw[->] (233) [red, dashed]--node[red]{$\times$}node[left,red]{$1$}  (133); 
    \draw[->] (223) [red,dashed]--node[red]{$\times$}  node[right,red]{$1$}(123); 
    \draw[->] (333) [green, double]-- node[right,green]{$0$}(233); 
    \draw[->] (233) [green,double]-- node[right,green]{$0$}(223); 
    \draw[->] (133) [green,double]-- node[left,green]{$0$}(123); 
    \end{tikzpicture}

    \caption{$w = [213]$}
\end{subfigure}
\hfill
\begin{subfigure}{0.3\textwidth}
    \centering
    \begin{tikzpicture}
    \node (333) at (0,6) {$\hleft 333  \hright$};
    \node (233) at (0,4.5) {$\hleft 233  \hright$};
    \node (223) at (1.5,3) {$\hleft 223  \hright$};
    \node (133) at (-1.5,3) {$\hleft 133  \hright$};
    \node (123) at (0,1.5) {$\hleft 123  \hright$};
    
    \draw[->] (233) [green,double]--  node[left,green]{$0$}(133); 
    \draw[->] (223) [green,double]-- node[right,green]{$0$}(123); 
    \draw[->] (333) [green,double]-- node[right,green]{$0$}(233); 
    \draw[->] (233) [green,double]-- node[right,green]{$0$}(223); 
    \draw[->] (133) [green,double]-- node[left,green]{$0$}(123); 
    \end{tikzpicture}
    
    \caption{$w=[132]$}
\end{subfigure}
\hfill
\begin{subfigure}{0.3\textwidth}
    \centering
    \begin{tikzpicture}
    \node (333) at (0,6) {$\hleft 333  \hright$};
    \node (233) at (0,4.5) {$\hleft 233  \hright$};
    \node (223) at (1.5,3) {$\hleft 223  \hright$};
    \node (133) at (-1.5,3) {$\hleft 133  \hright$};
    \node (123) at (0,1.5) {$\hleft 123  \hright$};
    
    \draw[->] (233) [green,double]--node[left,green]{$0$} (133); 
    \draw[->] (223) [green,double]--node[right,green]{$0$} (123); 
    \draw[->] (333) [green,double]--node[right,green]{$0$} (233); 
    \draw[->] (233) [green,double]-- node[right,green]{$0$}(223); 
    \draw[->] (133) [green,double]-- node[left,green]{$0$}(123); 
    \end{tikzpicture}
    \caption{$w=[123]$}
\end{subfigure}
\caption{Equations cutting the Hessenberg varieties at Schubert cell $C_w$ associated to Jordan type $(2,1)$.}
\label{fig: diagrams for lambda=(2,1)}
\end{figure}

%% file: Thm1/example_2_1_table.tex
\begin{align*}
\end{align*}

\begin{figure}[ht!]
    \centering
    $\begin{array}{|c|c||c||c||c|c||c|}
    \hline 
        \multicolumn{2}{|c|}{} &\multicolumn{5}{|c|}{h}\\
         \multicolumn{2}{|c|}{w \qquad   \text{matrix}} & \hleft 333  \hright & \hleft 233  \hright &  \hleft 133  \hright  & \hleft 223  \hright & \hleft 123  \hright \\ \hline \hline 
         [321] & \begin{bsmallmatrix}
             a_{11} & a_{12} & 1 \\
             a_{21} & 1 & 0  \\
             1  & 0 & 0
         \end{bsmallmatrix} & 0 & a_{21} & a_{21}  & 1 & 1  \\ \hline  \hline
         [231] & \begin{bsmallmatrix}
             a_{11} & a_{12} & 1 \\
             1 & 0 & 0  \\
             0  & 1 & 0
         \end{bsmallmatrix} & 0 & 1 & 1 & 1 & 1  \\ \hline 
         [312] & \begin{bsmallmatrix}
             a_{11} & 1 & 0 \\
             a_{21} & 0 & 1  \\
             1  & 0 & 0
         \end{bsmallmatrix} & 0 &  0 & a_{21}  & 0 & a_{21}  \\ \hline  \hline
         [213] & \begin{bsmallmatrix}
             a_{11} & 1 & 0 \\
             1 & 0 & 0  \\
             0  & 0 & 1
         \end{bsmallmatrix} & 0 & 0  & 1  & 0 & 1  \\ \hline
         [132] & \begin{bsmallmatrix}
             1 & 0 & 0 \\
             0 & a_{22} & 1 \\
             0  & 1 & 0
         \end{bsmallmatrix}  & 0  & 0  & 0  & 0 & 0  \\ \hline \hline
         [ 123 ] \hright & \begin{bsmallmatrix}
             1 & 0 & 0 \\
             0 & 1 & 0  \\
             0 & 0 & 1
         \end{bsmallmatrix}  & 0  &  0 &  0 & 0 & 0  \\ \hline 
    \end{array}$
    
    \caption{Table of equations cutting out the Hessenberg varieties at each Schubert cell $C_w$ for Jordan type $(2,1)$} 
    \label{table: (2,1)} 
\end{figure}

%% file: Julianna_connection/Tymoczko_intro.tex
In this section, we review Tymoczko's results and proof strategies regarding the existence of a paving by affines for Hessenberg varieties. 
\begin{theorem}[{\cite[Corollary 6.3]{Tymoczko2006}}]
\label{thm: Cor 6.3}
The Hessenberg variety $\mathcal{H}(X, h)$ is paved by affines for all $X$ and $h$. More specifically, we can choose $X$ within its conjugacy class so that:
\begin{enumerate}
    \item $C_w \cap \mathcal{H}(X, h) \neq \emptyset$ if and only if $w^{-1} X w \in H(h)$, and
\item If $C_w \cap \mathcal{H}(X, h) \neq \emptyset$ then $C_w \cap \mathcal{H}(X, h) \simeq \mathbb{C}^{d_w}$ for some nonnegative integer $d_w$. 
\end{enumerate}
\end{theorem}
\noindent To prove this result, Tymoczko used a particular type of a nilpotent operator, called the \emph{highest form}. To define this notion,
we introduce the following map $\phi$ on $\{1,\cdots, n\}$, which describes the row-indices of the pivots.
Assume that $X$ is an $n\times n$ nilpotent matrix whose entries take value either 0 or 1.
We denote by $\phi$ the map sending the column index $i$ to the row-index of the $i$-th column pivot of $X$, with the notation that $\phi(i) = 0$ if the $i$-th column is the zero vector. Explicitly, $\phi$ is determined by the following assignment:
    \begin{equation}\begin{cases}
    \label{eq: definition of phi}
        Xe_i = e_{\phi(i)} &\text{if }e_i \notin \ker X , \\
        \phi(i) = 0 &\text{if }e_i \in \ker X.
    \end{cases}\end{equation}
Using $\phi$, a nilpotent operator in \emph{highest form} is the matrix whose entries take value either 0 or 1, and the row indices of the pivots  $\{\phi(i)\}_{i=1}^n$ form a non-decreasing sequence.  

In this paper, we use a stronger notion, called a matrix in \emph{strictly highest form}. 
\begin{definition}
    \label{def: highest form}
    A nilpotent matrix is in \emph{strictly highest form} if its entries take value 0 or 1, and  the row indices of its pivots $\{\phi(i)\}_{i=1}^n$ form a strictly increasing sequence once positive. 
\end{definition}

\begin{ex}
\label{ex: highest form of (2,2)}
The Jordan canonical form $\begin{bsmallmatrix}
            0 & 1 &  0 & 0 \\
            0 & 0 &  0 & 0 \\
            0 & 0 &  0 & 1 \\
            0 & 0 &  0 & 0 
        \end{bsmallmatrix}$ of Jordan type $(2,2)$ is not in  highest form. On the other hand, the matrix $\begin{bsmallmatrix}
            0 & 0 &  1 & 0 \\
            0 & 0 &  0 & 1 \\
            0 & 0 &  0 & 0 \\
            0 & 0 &  0 & 0 
        \end{bsmallmatrix}$ is also of Jordan type $(2,2)$, and it's in strictly highest form. The matrix $\begin{bsmallmatrix}
            0 & 0 &  1 & 1 \\
            0 & 0 &  0 & 0 \\
            0 & 0 &  0 & 0 \\
            0 & 0 &  0 & 0 
        \end{bsmallmatrix}$ is in highest form but not in strictly highest form.
\end{ex}

\begin{remark}
    \label{rmk: assume phi is increasing}
    Strictly highest forms induce the following nice property on indices. If $X$ is in strictly highest form $f_{(k,m)}$ is a nonzero equation obtained by Theorem \ref{thm: Thm 1}, then the row-pivot index $p=\phi(w(k))$ of $Xv_k$ must be positive. In particular, by the increasing property of $\phi$ once positive, we obtain a well-defined relation $k = w^{-1}(\phi^{-1}(p))$. Therefore, in such a case, we may interchangeably work with variables $(p, r) = (\phi(w(k)), w(m))$ instead of $(k,m)$.
\end{remark}

\begin{remark}
    The definition of highest form stated above is stronger than that introduced by Tymoczko in \cite{Tymoczko2006}. Nevertheless, by Remark \ref{rmk: H(X,h) depend only on conjugacy class} and the following Lemma \ref{lem: JCF to HF}, all of the statements involving \emph{Hessenberg varieties} associated with nilpotent operators in Section \ref{sec: rederivation} remain valid whether we assume $X$ is in highest form in the sense of \cite{Tymoczko2006} or we assume $X$ is in strictly highest form.
\end{remark}

Dewitt and Harada \cite{dewitt2012poset} provided a construction of a nilpotent matrix in highest form for any Jordan type. Moreover, one can verify that the constructed matrix is also in strictly highest form. The following lemma is a simplified version of the construction.


\begin{lemma} [{\cite[Theorem 3.21]{dewitt2012poset}}]
    \label{lem: JCF to HF}
    Suppose that $X$ is the Jordan canonical form of type $\lambda = (\lambda_1,\cdots, \lambda_k)$ in non-decreasing order, i.e.,  $\lambda_i \geq \lambda_{i+1}$ for all $i$. Consider the Young diagram $Y$ associated to $\lambda$, and fill in the numbers $1,\cdots, n$ from bottom to top in each column, from left to right. Denote the value of the Young tableau at $(i,j)$ by $Y_{i,j}$. Then the matrix $X^h$ sending $e_{Y_{i,j}}$ to $e_{Y_{i,j-1}}$ for $j > 1$ and to zero for $j = 1$ is in (strictly) highest form. 
\end{lemma}



\begin{ex}
Let $X$ be a nilpotent matrix of type $\lambda = (3, 2, 1)$ given by
$$
X = \begin{bmatrix}
    0 & 1 & 0 & 0 & 0 & 0 \\
    0 & 0 & 1 & 0 & 0 & 0 \\
    0 & 0 & 0 & 0 & 0 & 0 \\
    0 & 0 & 0 & 0 & 1 & 0 \\
    0 & 0 & 0 & 0 & 0 & 0 \\
    0 & 0 & 0 & 0 & 0 & 0 
\end{bmatrix}.
$$
This matrix acts on the standard basis vectors as
$$
e_3 \mapsto e_2 \mapsto e_1 \mapsto 0, \quad e_5 \mapsto e_4 \mapsto 0, \quad \text{and} \quad e_6 \mapsto 0.
$$
    We encode these data into the Young diagram associated with the Jordan type  $(3,2,1)$ as follows: we read the labels from right to left, which induces the mapping of the standard basis vectors with the corresponding labels.
    \begin{align*}\begin{ytableau}
3 & 5 & 6 \\
2 & 4 \\
1
\end{ytableau} \qquad \leadsto \qquad 
    \begin{cases}
    0  \mapsfrom e_3 \mapsfrom e_5 \mapsfrom e_6 \\
    0 \mapsfrom e_2 \mapsfrom e_4 \\
    0  \mapsfrom e_1.
    \end{cases}
    \end{align*}
    The matrix corresponding to this operation is in highest form, given by
    \begin{align*}
        \begin{bmatrix}
            0 & 0 & 0 & 0 & 0 & 0 \\
            0 & 0 & 0 & 1 & 0 & 0 \\
            0 & 0 & 0 & 0 & 1 & 0 \\
            0 & 0 & 0 & 0 & 0 & 0 \\
            0 & 0 & 0 & 0 & 0 & 1 \\
            0 & 0 & 0 & 0 & 0 & 0 
        \end{bmatrix}.
    \end{align*}
\end{ex}

Using the highest form, Tymoczko also provided a formula for the dimension $d_w$ of the affine paving, as follows.
\begin{theorem}[{\cite[Corollary 6.3]{Tymoczko2006}}]
\label{thm: Tymoczko's dw}
Assume that $X$ is in the highest form. Then the dimension $d_w$ of the Schubert cell  $C_w \cap \mathcal{H}(X, h) $ is given by 
    $$d_w = \left|\left\{(i, k) | k > i, w^{-1}(i) > w^{-1}(k), \text{ and } h(w^{-1}(j)) \geq w^{-1}(i) \text{ if $X_{kj}$ is nonzero}\right\}\right|.$$
\end{theorem}


%% file: Julianna_connection/highest_form.tex
\label{sec: affineness of strata}

First, we rederive Statement (2) of Theorem \ref{thm: Cor 6.3}, on the affineness condition of the nonempty intersection $C_w \cap \cH(X,h)$ when the nilpotent operator $X$ is in strictly highest form. The proof strategy is as follows. Let $S_{MT}$ be the collection of nonzero equations $f_{(k,m)}$ cutting out $C_w \cap \cH(X,h)$ in $C_w$, which are obtained by applying the Main Theorem Theorem \ref{thm: Thm 1} to each edge of a fixed Hessenberg path $h$ to the top Hessenberg function $h^{\operatorname{top}}$. Because $C_w$ is an affine space, to show that a nonempty intersection  $C_w \cap \cH(X,h)$ is an affine space, it suffices to show that there is a total ordering $f_{(k_1,m_1)} \prec f_{(k_2,m_2)} \prec \cdots $ of equations in $S_{MT}$ with the property that for each $i$, cutting the intermediate subscheme $V(f_{(k_1,m_1)},\cdots, f_{(k_{i-1},m_{i-1})})$ in $C_w$ by the equation $f_{(k_i,m_i)}$ is equivalent to eliminating a free variable of the coordinate ring of the intermediate subscheme. The following Lemma, also introduced as Lemma \ref{intro:lem1} in the Introduction, describes  the conditions on the strict total order $\prec$ for which this property holds in more generality.

\begin{lemma}[Also introduced as Lemma \ref{intro:lem1}]
\label{lem: reduction of affineness to total order}
Assume that  $X$ is a nilpotent operator and the intersection $C_w \cap \mathcal{H}(X,h)$ is nonempty. Let $S$ be a finite set of nonzero equations cutting $C_w \cap \mathcal{H}(X,h)$ in $C_w$.
To show $C_w \cap \mathcal{H}(X,h)$ is an affine space, it suffices to show the following two properties: 
\begin{enumerate}
    \item for each $f_\alpha \in S$, we can associate a Schubert variable $a^{(\alpha )}$ such that $f_{\alpha}$ can be decomposed as 
    $f_{\alpha} = -a^{(\alpha)} + g_{\alpha},$
    where $g_{\alpha}$ is independent of $a^{(\alpha)}$, and
    \item there exists a strict total order $\prec$ of equations in $S$ such that the variable $a^{(\alpha)}$ does not appear in any of the equations succeeding $f_{\alpha}$ in $S$.
\end{enumerate}
Moreover, if Properties (1) and (2) hold, then the codimension of $C_w \cap \cH(X,h)$ in $C_w$ is the cardinality of $S$.
\end{lemma}

\begin{proof}
    Suppose that Properties (1) and (2) hold. Let $A_0$ be the collection of Schubert variables $a^{(\alpha)}$, where $f_{\alpha} \in S$. Denote the minimal element of $S$ with respect to the strict total order $\prec $ by $f_{\alpha_1}$, and write
    \begin{align*}
        S = S_1 \sqcup \{f_{\alpha_1}\}, \quad A_0 = A_1 \sqcup \{a^{(\alpha_1)}\}.
    \end{align*}
    By the property of the strict total order, all the equations in $S_1$ do not depend on the variable $a^{(\alpha_1)}$.  
    Then the coordinate ring of $C_w \cap \mathcal{H}(X,h)$ is given by 
    \begin{align*}
        k[A_0]/(S) &\cong k[A_1][a^{(\alpha_1)}]\Big/\big((S_1)+(f_{(\alpha_1)})\big) \\
        &\cong \Big(k[A_1]/(S_1)\Big)[a^{(\alpha_1)}]\Big/(-a^{(\alpha_1)} + g_{\alpha_1}\big) 
        \cong k[A_1]/(S_1),
    \end{align*}
    where by $k[A_i]$ we mean the coordinate ring with generators $a^{(\alpha)} \in A_i$.
    Inductively, by applying the same argument  to the  (unique) minimal element of $S_1$ with respect to the total order, we obtain 
    \begin{align*}
        k[A_0]/(S) \cong k[A_1]/(S_1) \cong \cdots \cong k[A_i]/(S_i),
    \end{align*}
    where $S_i$ is the complement of the first $i$ equations in the total order, and $A_i$ is the complement of the corresponding $i$ variables associated with the $i$ equations in the first order. In particular, by setting $i$ to be the cardinality of $S$, we obtain that the coordinate ring is an affine space of codimension $|S|$.
\end{proof}

\begin{remark}
    For the remaining of the section, we are concerned with the case where $S$ is the collection $S_{MT}$ of nonzero equations cutting out $C_w \cap \cH(X,h)$ in $C_w$, obtained by applying the Main Theorem Theorem \ref{thm: Thm 1} to each edge of a fixed Hessenberg path $h$ to the top Hessenberg function $h^{\operatorname{top}}$.
\end{remark}

To illustrate the strategy for constructing a total order on $S_{MT}$ that satisfy Properties (1) and (2) above, we start with the following example. 
\begin{ex} 
    \label{ex: strict total order of f_(k,m)}
    Consider a nilpotent operator of Jordan type $(4)$ and the Schubert cell associated with $w=[4321]$. The corresponding highest form nilpotent operator $X$ and the canonical Schubert form $V$ are given by 
    $$X = \begin{bmatrix}
        0 & 1 & 0 & 0 \\
        0 & 0 & 1 & 0 \\
        0 & 0 & 0 & 1 \\
        0 & 0 & 0 & 0 
    \end{bmatrix}, \quad  V = 
    \begin{bmatrix}
        a & d & f & 1 \\
        b & e & 1 & 0 \\
        c & 1 & 0 & 0 \\
        1 & 0 & 0 & 0 
    \end{bmatrix}, \quad \text{so that }X V = \begin{bmatrix}
        b & e & \textbf{1} & 0 \\
        c & \textbf{1} & 0 & 0 \\
        \textbf{1} & 0 & 0 & 0 \\
        0 & 0 & 0 & 0 
    \end{bmatrix}.$$ 
    By expanding $Xv_k$ in terms of column vectors of $V$,  we obtain equations $f_{(k,m)}$:
\begin{alignat*}{4}
    f_{(1,1)} &= 0, \qquad & f_{(2,1)} &= 0, \qquad & f_{(3,1)} &= 0, \qquad & f_{(4,1)} &= 0 ,\\
    f_{(1,2)} &= \textbf{1} ,& f_{(2,2)} &= 0, & f_{(3,2)} &= 0, & f_{(4,2)} &= 0, \\
    f_{(1,3)} &= c-e ,& f_{(2,3)} &= \textbf{1} ,& f_{(3,3)} &= 0, & f_{(4,3)} &= 0, \\
    f_{(1,4)} &= b -d - (c-e)f, \quad & f_{(2,4)} &= e-f, \quad  & f_{(3,4)} &= \textbf{1} , \quad & f_{(4,4)} &= 0.
\end{alignat*}

Note that we are concerned with nonzero non-unit equations which can be obtained by applying Theorem \ref{thm: Thm 1}, which are $f_{(1,3)}, f_{(1,4)}$, and $f_{(2,4)}$.
To visualize the Schubert variables appearing in equations $f_{(k,m)}$, we highlight in yellow the variables appearing in $f_{(k,m)}$ in the canonical Schubert matrix. For example, the equation $f_{(1,4)}$ has variables $b,d,c,e,f$ in its expansion. Hence the variables in the Schubert cell are highlighted as follows:
    
\begin{figure}[htbp]
  \centering
  \begin{tikzpicture}[>=stealth]
    \matrix [
      matrix of math nodes,
      left delimiter={[},
      right delimiter={]},
      nodes={
        anchor=center,
        inner sep=1pt,
        minimum size=0pt 
      },
      column sep=0.5em,
      row sep=0.4em
    ] (m) {
      a & d & f & 1 \\
      b & e & 1 & 0 \\
      c & 1 & 0 & 0 \\
      1 & 0 & 0 & 0 \\
    };

    \begin{scope}[on background layer]
      \def\highlightradius{2.8mm}
      
      \fill[yellow!50] (m-1-2.center) circle (\highlightradius);
      
      \foreach \pos in {2-1, 3-1, 2-2, 1-3} {
        \fill[yellow!50] (m-\pos.center) circle (\highlightradius);
      }
    \end{scope}
  \end{tikzpicture}
  \caption{Variables occurring in $f_{(1,4)}$ highlighted in yellow.}
\end{figure}
\noindent Applying the same visualization to the nontrivial equations $f_{(2,4)}$ and $f_{(1,3)}$, we obtain the diagram in Figure \ref{fig: hilighted variables in schubert cells for all equations}. Let $a^{(1,4)} = d$, $a^{(2,4)} = f$, and $a^{(1,3)} = e$ be the Schubert variables associated with $f_{(1,4)},$ $ f_{(2,4)}$, and $f_{(1,3)}$, respectively.
\begin{figure}[ht!]
  \centering
  \setlength{\tabcolsep}{15pt} 
  \begin{tabular}{ccc}
    \begin{tikzpicture}[baseline=(m.center), >=stealth]
      \matrix [
        matrix of math nodes,
        left delimiter={[},
        right delimiter={]},
        nodes={anchor=center, inner sep=1pt, minimum size=0pt, font=\small},
        column sep=0.5em,
        row sep=0.3em
      ] (m) {
        a & d & f & 1 \\
        b & e & 1 & 0 \\
        c & 1 & 0 & 0 \\
        1 & 0 & 0 & 0 \\
      };
      \begin{scope}[on background layer]
        \def\highlightradius{2.2mm}
        \fill[yellow!50] (m-1-2.center) circle (\highlightradius);
        \draw[red!80!black, thick] (m-1-2.center) circle (\highlightradius);
        \foreach \pos in {2-1, 3-1, 2-2, 1-3} {
          \fill[yellow!50] (m-\pos.center) circle (\highlightradius);
        }
      \end{scope}
    \end{tikzpicture} 
    & 
    \begin{tikzpicture}[baseline=(m.center), >=stealth]
      \matrix [
        matrix of math nodes,
        left delimiter={[},
        right delimiter={]},
        nodes={anchor=center, inner sep=1pt, minimum size=0pt, font=\small},
        column sep=0.5em,
        row sep=0.3em
      ] (m) {
        a & d & f & 1 \\
        b & e & 1 & 0 \\
        c & 1 & 0 & 0 \\
        1 & 0 & 0 & 0 \\
      };
      \begin{scope}[on background layer]
        \def\highlightradius{2.2mm}
        \fill[yellow!50] (m-1-3.center) circle (\highlightradius);
        \draw[red!80!black, thick] (m-1-3.center) circle (\highlightradius);
        \foreach \pos in {2-2} {
          \fill[yellow!50] (m-\pos.center) circle (\highlightradius);
        }
      \end{scope}
    \end{tikzpicture}
    & 
    \begin{tikzpicture}[baseline=(m.center), >=stealth]
      \matrix [
        matrix of math nodes,
        left delimiter={[},
        right delimiter={]},
        nodes={anchor=center, inner sep=1pt, minimum size=0pt, font=\small},
        column sep=0.5em,
        row sep=0.3em
      ] (m) {
        a & d & f & 1 \\
        b & e & 1 & 0 \\
        c & 1 & 0 & 0 \\
        1 & 0 & 0 & 0 \\
      };
      \begin{scope}[on background layer]
        \def\highlightradius{2.2mm}
        \fill[yellow!50] (m-2-2.center) circle (\highlightradius);
        \draw[red!80!black, thick] (m-2-2.center) circle (\highlightradius);
        \foreach \pos in {3-1} {
          \fill[yellow!50] (m-\pos.center) circle (\highlightradius);
        }
      \end{scope}
    \end{tikzpicture} \\
    \vspace{0.5mm} & \vspace{0.5mm} & \vspace{0.5mm} \\
    $f_{(1,4)} = b-d-(c-e)f$ & $f_{(2,4)} = e-f$ & $f_{(1,3)} = c-e$
  \end{tabular}
  \caption{Schubert variables occurring in equations $f_{(k,m)}$ highlighted in yellow; circled in red is the Schubert variable $a^{(k,m)}$ associated with $f_{(k,m)}$.}
  \label{fig: hilighted variables in schubert cells for all equations}
\end{figure}
From these diagrams, we observe that the variable $a^{(1,4)}$ associated with $f_{(1,4)}$ does not appear in either $f_{(1,3)}$ or $f_{(2,4)}$. Similarly, the variable $a^{(2,4)} $ associated with $f_{(2,4)}$ does not appear in $f_{(1,3)}.$ It follows that the order $f_{(1,4)} \prec f_{(2,4)} \prec f_{(1,3)}$ satisfies Properties (1) and (2) of Lemma \ref{lem: reduction of affineness to total order}.
\end{ex}

We make two key observations about this example. First, the associated Schubert variables are taken to be the second term that appears in the expansion of $f_{(k,m)}$ given by the recursive formula. Second, the lower row-index the associated Schubert variable has, the more preceding the corresponding equation should be in the total order.

Generalizing the observations, we give a general constructive procedure to associate a Schubert variable $a^{(k,m)}$ with each equation $f_{(k,m)}$ for any Jordan type and any Hessenberg function so that Properties (1) and (2) hold. To do this, we first note the following property of a strictly highest form operator: if the image of a column vector of the canonical Schubert matrix under $X$ is nonzero, then the image vector has an entry with value 1. In Example \ref{ex: strict total order of f_(k,m)}, this property is shown in bold.

\
\begin{lemma}
    \label{lem: pivot value 1}
    Let $X$  be a nilpotent operator in strictly highest form and $V$ the canonical Schubert matrix. Then every nonzero column of $XV$ has pivot value 1.
\end{lemma}

\begin{proof}
    Denote $d = \dim(\ker X)$. Then by construction of the strictly highest form, $e_1,\cdots, e_d \in \ker X$. Take the $k$th column $v_k$ of $V$. Note that because $V$ is in the canonical form of the Schubert cell, the pivot of $v_k$ has value $1$ for all $i$. We show that if $Xv_k \neq 0$, then the pivot value remains $1$. If the pivot index $p$ of $v_k$ is at most $d$, then $v_k$ is a linear combination of $e_1,\cdots, e_d$, so $Xv_k = 0$. Hence, suppose otherwise; if the pivot index $p$ of $v_k$ is more than $d$, then the pivot of $v_k$ is shifted above after the action of $X$ because $X$ is in strictly highest form. In particular, the pivot of $Xv_k$ is still 1.
\end{proof}


This property, along with the additional property that $f_{(k,m')} = 1$ for some $m' < m$, guarantees the existence of a canonical choice for the associated 
Schubert variable with Property (1).  
 
When $f_{(k,m')} = 1$ for some $m' < m$, this property above allows us to decompose $f_{(k,m)}$ as  $-a^{(k,m)} + g_{(k,m)}$, where $g_{(k,m)}$ is independent of $a^{(k,m)}$.


\begin{prop}
    \label{prop: f(k,m) has a linear term}
    Suppose that the nilpotent operator $X$ is in strictly highest form and $V$ is the canonical matrix form of the Schubert cell associated with $w \in S_n$.  Let $S_{MT}$ be the set of nonzero equations $f_{(k,m)}$ cutting out $C_w \cap \mathcal{H}(X,h)$ obtained by recursively applying Theorem \ref{thm: Thm 1} to a fixed Hessenberg path from $h$ to $h^{\operatorname{top}}$. Then Property (1) of Lemma \ref{lem: reduction of affineness to total order}  holds for $S=S_{MT}$.

    More explicitly, let  $f_{(k,m)}$ be the equation obtained by the recursive equation
    \begin{align*}
    f_{(k,m)}  =  \Big[Xv_k - \sum_{w(j)> w(m)}  f_{(k,j)} v_{j}\Big]_{w(m)}
    \end{align*}
    given in Theorem \ref{thm: Thm 1}. Assume $f_{(k,m)} \neq 0 ,1$.
    If $f_{(k,m')} = 1$ for some $m' < m$ (or equivalently, $w^{-1}(\phi(w(k))) < m$, where $\phi$ is given by (Equation \ref{eq: definition of phi})), then the Schubert variable $a^{(k,m)} := a_{w(m), w^{-1}(\phi(w(k)))}$  has the property that $f_{(k,m)}$ decomposes as  $-a^{(k,m)} + g_{(k,m)}$, where $g_{(k,m)}$ is a polynomial independent of $a^{(k,m)}$ and has no constant terms. 
\end{prop}

We call $a^{(k,m)}$ the \emph{Schubert variable associated with} $f_{(k,m)}$.

\begin{proof}
    We reframe the problem statement as follows. Recall the proof idea of  Theorem \ref{thm: Thm 1}, which is to expand $Xv_k$ in terms of the basis vectors in the order of pivot positions, $\{v^{(n)},\cdots, v^{(1)}\}$, where $v^{(q)} = v_{w^{-1}(q)}$. Denote  the pivot index of $Xv_k$ by $p = \phi(w(k))$ and also set $r = w(m)$.The vector $Xv_k$ is expanded uniquely in terms of the basis $\{v^{(q)}\}_{q=1}^{n}$ as
    $$Xv_k = c^{(p)}v^{(p)} + \cdots +c^{(1)}v^{(1)},$$
    where $f_{(k,m)} = c^{(r)}$.

    By Remark \ref{rmk: assume phi is increasing}, we may work with variables $p, r$ in place of $k,m$.
    Because $c^{(r)}$ is assumed to be nonzero, we have $p > r > 0$. Moreover, by Lemma \ref{lem: pivot  value 1}, the condition $f_{(k,m')} = 1$ for $m' < m$ is equivalent to $w^{-1}(p) < w^{-1}(r)$  (i.e., $w^{-1}(\phi(w(k)) < m$ as given in the statement).  Moreover, $$a^{(k,m)} := a_{w(m),w^{-1}(\phi(w(k)))} = [v^{(p)}]_r.$$ Therefore, we wish to show that if $r < p$ and $w^{-1}(p) < w^{-1}(r)$, then $c^{(r)}$ contains a term of the form $-[v^{(p)}]_r$, and the variable $a_{r, w^{-1}(p)}$ does not appear in other terms of $c^{(r)}$.

    By the derivation of the formula, we have
    $$c^{(r)} =    [Xv_k]_r - [v^{(p)}]_r - \sum_{p > q > r}  c^{(q)} [v^{(q)}]_{r} \text{ for }r < p,$$
    where both $c^{(q)}$ and $[v_{(q)}]_p$ have no constant terms for $r < q < p$.
    Then we have a decomposition
    $$f_{(k,m)} = -a^{(k,m)} + g_{(k,m)}, \quad \text{ where } g_{(k,m)} = [Xv_k]_r - \sum_{p > q > r} c^{(q)}[v^{(q)}]_r.$$
    The last summation term in $g_{(k,m)}$, if nonzero, must be at least quadratic in the Schubert variables.  Moreover, by the recursive relation of $c^{(q)}$, the variable $[v^{(p)}]_r$ cannot appear in the expansion of  $c^{(q)}[v^{(q)}]_r$ for all $q > r$. Therefore, to show the claim, it remains to show that $[Xv_k]_r $ is independent of the variable $a^{(k,m)}$ and has no constant term.  These claims follow from the observation that $X$ is in the highest form and $v_k = v^{(w(k))}\neq v^{(\phi(w(k)))} = v^{(p)}$, so $[Xv_k]_r$ is either zero or a Schubert variable that is not $a^{(k,m)}$.
\end{proof}

\begin{remark}
    Observe that when $f_{(k,m)} \in S_{MT}$ and $f_{(k,m')} = 1$ for some $m' \geq m$, then $S_{MT}$ would contain the equation $f_{(k,m')}$, so $C_w \cap \cH(X,h)$ would be empty. 
    Thus, the hypothesis imposed in the Proposition that $f_{(k,m')} = 1$ for $m' < m$ will not interfere with the later discussion when we impose vanishing of $f_{(k,m)}$ to eliminate the Schubert variable $a^{(k,m)}$ in the coordinate ring of $C_w \cap \cH(X,h)$. 
\end{remark}

\begin{remark}
    The hypothesis that $X$ is in strictly highest form is crucial to the validity of Proposition \ref{prop: f(k,m) has a linear term}. For a counterexample when this hypothesis is dropped, consider the Jordan canonical form with Jordan type $\lambda = (1,2,1)$ and  the Schubert cell associated with $w = [4321]$. The corresponding matrices are given by 
    $$X = \begin{bmatrix}
        0 & 0 & 0 & 0 \\
        0 & 0 & 1 & 0 \\
        0 & 0 & 0 & 0 \\
        0 & 0 & 0 & 0 
    \end{bmatrix}, \quad  V = 
    \begin{bmatrix}
        a & d & f & 1 \\
        b & e & 1 & 0 \\
        c & 1 & 0 & 0 \\
        1 & 0 & 0 & 0 
    \end{bmatrix}, \quad \text{so that }X V = \begin{bmatrix}
        0 & 0 & 0 & 0 \\
        \textbf{c} & \textbf{1} & 0 & 0 \\
        0 & 0 & 0 & 0 \\
        0 & 0 & 0 & 0 
    \end{bmatrix}.$$ 
    Observe that $Xv_1$ has pivot value $c$, not $1$.
    It follows that $Xv_1$ is expanded in the basis $\{v_i\}_{i=1}^4$ as $Xv_1 = cv_3 - cf v_4$, which gives $f_{(1,4)} = cf$. It follows that the intersection $C_w \cap \cH(X, \hleft 1244 \hright)$ is the closed subscheme of $C_w$ cut out by the equation $cf$, which is not affine. 
\end{remark}

Next, we show that there exists a strict total ordering  on $S_{MT}$ satisfying Property (2) in Lemma \ref{lem: reduction of affineness to total order}.

\begin{prop}
    \label{prop: existence of total order}
    Assume that the nilpotent operator $X$ is in its strictly highest form.
    Suppose that  $C_w \cap \mathcal{H}(X,h)$ is nonempty. Let $S_{MT}$ be the set of nonzero equations $f_{(k,m)}$ cutting out $C_w \cap \mathcal{H}(X,h)$ obtained by recursively applying Theorem \ref{thm: Thm 1} to a fixed Hessenberg path from $h$ to $h^{\operatorname{top}}$.
    Then property (2) of Lemma \ref{lem: reduction of affineness to total order} holds on $S=S_{MT}$.
    
    More precisely, there is a strict total ordering $\prec$ on $S_{MT}$ such that if $f_{(k,m)} \prec f_{(k', m')}$, then the Schubert variable $a^{(k,m)}$ associated with $f_{(k,m)}$ does not appear in $f_{(k',m')}.$ One such strict total ordering is given by
    \begin{align*}
        f_{(k,m)} \prec f_{(k',m')} \overset{\text{def}}{\iff} \begin{cases}
        w(m) < w(m'),  \text{ or }\\
        m = m' \text{ and }w(k) > w(k').
        \end{cases}
    \end{align*}
    
\end{prop}

\begin{proof}

    We prove a stronger version of the claim: the relation $\prec$ given in the statement above is a strict total order on the set $S_{MT}^{\operatorname{tot}}$ of all nonzero and non-unit equations $f_{(k,m)}$ obtained by the Main Theorem. If this is true, then for any subset $S$ of $S_{MT}^{\operatorname{tot}}$, $\prec$ remains a strict total order on $S$, and the claim follows.
    
    By the nonzero assumption on $f_{(k,m)}$ and Remark \ref{rmk: assume phi is increasing}, we may work with variables $(r,p) = (w(m), \phi(w(k))$ and similarly $(r',p') = (w(m'), \phi(w(k')))$. Moreover, use the notation $v^{(q)} = v_{w^{-1}(q)}$; then the Schubert variable associated with $f_{(k,m)} = c^{(r)}$ is $a^{(k,m) }= [v^{(p)}]_r$. Then the goal is to show that 
    \begin{align*}
        \begin{cases}
            r < r', \qquad  & (\text{Case } 1) \quad \text{ or } \\
            r = r' \text{ and } \phi^{-1}(p) > \phi^{-1}(p') & (\text{Case } 2)
        \end{cases}
    \end{align*}
    implies that the variable $[v^{(p)}]_r$ does not appear in the expansion of $c^{(r')}$. 
    From the recursive formula 
    $$f_{(k',m')} = c^{(r')} =[Xv_{k'}]_{r'} - [v^{(p')}]_{r'} - \sum_{p' > q > r'}  c^{(q)} [v^{(q)}]_{r'} \text{ for }r' \leq  p', $$
    the variable $[v^{(p)}]_r$ occurs in the expansion of $f_{(k',m')} $    
    only if it occurs in the form of $[Xv_{k'}]_q = [v^{(\phi^{-1}(p'))}]_{\phi^{-1}(q)}$ when $q \in \operatorname{im}(\phi)$ and $p' > q \geq r'$, or $[v^{(q)}]_{r'} $ for $p' \geq q > r'$. Because the Schubert variables are determined by their positions in the Schubert canonical form, the variable $[v^{(p)}]_r$ occurs in the expansion of $f_{(k',m')}$ only if
    \begin{align*}
        (p,r) &= \begin{cases}
            (\phi^{-1}(p'), \phi^{-1}(q)) &\text{ for some } r' \leq  q < p' \text{ with }q \in \operatorname{im}(\phi), \text{ or }\\
            (q, r') &\text{ for some } r' < q \leq p'
        \end{cases} \\
        &\iff  \begin{cases}
            r' \leq \phi(r) < p' \text{ and } \phi(p) = p',&\text{(Condition A)}\quad \text{ or }  \\
            r=r' \text{ and }  r < p \leq p'&\text{(Condition B)}
        \end{cases}
    \end{align*}
    We show that neither of these conditions happen in both Case 1 and Case 2.
    In Case 1 ($r < r'$), from $\phi(r) < r < r'$, Condition A is not satisfied. Also, $r = r'$ in Condition B cannot be satisfied. Next, assume Case 2 ($r = r'$ and $\phi^{-1}(p) > \phi^{-1}(p')$): because $\phi(r) = \phi(r') < r'$,  Condition A cannot be satisfied. Also, by the property of $\phi$ associated to the strictly highest form, $\phi^{-1}(p) > \phi^{-1}(p')$ implies $p > p'$; hence Condition B cannot be satisfied.
    Therefore, it follows that when Case 1 or Case 2 holds, $[v^{(p)}]_r$ cannot occur in the expansion of $f_{(k',m')}$, i.e., $f_{(k,m)} \prec f_{(k',m')}$.

    Finally, because the given relation on $f_{(k,m)}$
    is lexicographical order on the pairs of values $(r,\phi^{-1}(p)) = (w(m), w(k))$, where the order is increasing in $r$ and decreasing in $\phi^{-1}(p)$, it is a strict total order on $S^{\operatorname{tot}}_{MT}$.
\end{proof}

Altogether,  affineness of the intersection $C_w \cap \cH(X,h)$ follows.
\begin{theorem}
    \label{thm: rederivation of affineness}
    Assume that the nilpotent operator $X$ is in strictly highest form.
    For each $w$ and $h$, $C_w \cap \mathcal{H}(X,h)$ is either an affine space or empty. 
\end{theorem}

\begin{proof}
    Assume $C_w \cap \mathcal{H}(X,h) $ is nonempty. By Lemma \ref{lem: reduction of affineness to total order}, it suffices to show Properties (1) and (2) hold. Property (1) holds by Proposition \ref{prop: f(k,m) has a linear term},  and Property (2) holds by Proposition \ref{prop: existence of total order}. 
\end{proof}

\subsection{Dimension of the intersection $C_w \cap \cH(X,h)$}
\label{sec:dimstrata}
Next, using the Main Theorem, we compute the codimension of the nonempty intersection $C_w \cap \cH(X,h)$ in $C_w$ to prove Theorem \ref{thm: Tymoczko's dw}.  By Lemma \ref{lem: reduction of affineness to total order}, it suffices to compute the cardinality $|S_{MT}|$, where
$$S_{MT} = \{f_{(k,m)} \neq 0  \ | \  h(k) < m \leq n ,\  k = 1,\cdots, n\}. $$
To compute this number, we identify the conditions for which $f_{(k,m)}$ with $h(k) < m$ for all $k$ is a nonzero equation.

\begin{lemma}
    Assume that the nilpotent operator $X$ is in highest form and $C_w \cap \cH(X,h)$ is nonempty. 
    Then $f_{(k,m)}$ is a nonzero equation if and only if $w(m) < \phi(w(k))$. 
\end{lemma}

\begin{proof} 
    Denote $p =\phi(w(k))$ and $r = w(m)$. First, because $C_w \cap \cH(X,h)$ is nonempty, we must have $w^{-1}(p) < w^{-1}(r)$ and $r \neq p$. Moreover, if $r > p$, then we have $f_{(k,m)} = 0$, so $f_{(k,m)} \neq 0$ implies $r < p$. It remains to show the converse. 
    Observe that the inequality $w^{-1}(p) < w^{-1}(r)$ implies that the $r$-th row pivot of the Schubert cell $V$ must appear in a later column than the $w^{-1}(p)$-th column. Hence we obtain
     $$[v^{(p)}]_{r}  = [v_{w^{-1}(p)}]_{r} = \begin{cases}
        \text{a Schubert variable }& \text{ if } r < p \\ 
        0 & \text{ if } r  > p  
    \end{cases}.$$
    By the recursive formula in Theorem \ref{thm: Thm 1}, $f_{(k,m)}$ can be written as a sum of three terms
    $$f_{(k,m)} = c^{(r)} = \Big[Xv_k - v^{(p)} - \sum_{p > q > r}  c^{(q)} v^{(q)}\Big]_{r}, $$
    and by the same argument as in the proof of Proposition \ref{prop: f(k,m) has a linear term}, the first and the last term of $f_{(k,m)}$ is independent of $[v^{(p)}]_r$, so it does not cancel the second term. It follows that if $r < p$, then $f_{(k,m)} \neq 0$.
\end{proof}

\begin{corollary}
    \label{cor: cardinality of |S|}
    If $C_w \cap \cH(X,h)$ is nonempty, then its codimension in $C_w$ is
    \begin{align*}
        |S_{MT}| &=\left | \left\{(k,m)  \ |\  h(k) < m ,  w(m) < \phi(w(k)) \right\}\right|.  \quad \qed
    \end{align*}
    
\end{corollary}
\begin{remark}
    In the description of $|S_{MT}|$ in Corollary \ref{cor: cardinality of |S|}, by construction, the pair $(k,m)$ correspond to the indices of the equation $f_{(k,m)}$. If we use 
    the change of variables associated with $a_{r,s} := a_{w(m), w^{-1}(\phi(w(k)))}$, then one can rewrite $|S_{MT}|$ in terms of the associated Schubert variable indices as
    $$|S_{MT}|= \left|\left \{(r,s) \ | \ h(w^{-1}(\phi^{-1}(w(s)))) < w^{-1}(r), \ r < w(s), \ s < w^{-1}(r)  \right\}\right|,$$
    where the last additional condition arose from the assumption that $a_{r,s}$ is a Schubert variable, so that pivots of the $r$-th row and the $s$-th column must be to the below and right of the matrix, respectively.
\end{remark}

\begin{remark}
    \label{rmk: |S| in (r,s)}
    We contrast this result with the original statement on the dimension $d_w$ given in Theorem \ref{thm: Tymoczko's dw} by Tymoczko. 
Using the definition of $\phi$ associated to the highest form nilpotent operator $X$ and $\dim(C_w) =\left \{(i,k)  |  i < k, w^{-1}(i) > w^{-1}(k)\right\}$, one computes that the codimension computation in Theorem \ref{thm: Tymoczko's dw} is
\begin{align*}
    \dim(C_w) - d_w &= \left|\left\{(i, k) | i < k, w^{-1}(i) > w^{-1}(k),   h(w^{-1}(\phi^{-1}(k))) < w^{-1}(i) \right\}\right |.
\end{align*}
After relabeling indices using $(r,s) = (i, w^{-1}(k))$, we obtain that the codimension computation agrees with that of $|S_{MT}|$ in Remark \ref{rmk: |S| in (r,s)}.
\end{remark}

%% file: Thm2/thm2.tex
\subsection{Nonemptiness condition of $C_w \cap \cH(X, h)$}
\label{sec: thm 2}

Finally, we rederive Statement (1) of Theorem \ref{thm: Cor 6.3}, the nonemptiness condition of the intersection of a Hessenberg variety $\cH(X,h)$ with a given Schubert cell $C_w$. 
We show that this is equivalent to the existence of a unique minimal Hessenberg function whose associated Hessenberg variety intersects nontrivially with a fixed Schubert cell $C_w$, where minimality is in the sense of partial ordering of Hessenberg functions.

\begin{theorem}[Also introduced as Theorem \ref{thm: Thm 2, intro}]
\label{thm: Thm 2}
Given a Schubert cell  $C_w$ and a nilpotent matrix $X$ in highest form, there exists a unique Hessenberg function $h_{\operatorname{min}}$ such that $C_w \cap \mathcal{H}(X, h)$ is nonempty if and only if $h \geq h_{\operatorname{min}}$, where $h_{\operatorname{min}}$ is minimal in the sense of the partial ordering of Hessenberg functions defined in Definition \ref{prop: poset structure on hessenberg functions}. Concretely, $h_{\operatorname{min}}$ is given by
\begin{align*}
    h_{\min}(k) = \begin{cases}
         \max\{k, h_{\operatorname{min}}(k-1)\} &\text{if  } w(k) \leq \dim(\ker(X)), \\
        \max\{k, h_{\min}(k-1), w^{-1}(\phi(w(k)))\}  &\text{if  } w(k) > \dim(\ker(X)),
    \end{cases}
\end{align*}
\noindent where we use the convention $h_{\operatorname{min}}(0) = 0$.
\end{theorem}

\begin{proof}
    To prove existence and uniqueness of the Hessenberg function with the desired property, we give a constructive algorithm to generate such a minimal Hessenberg function $h_{\min}$. The idea is as follows:  by Proposition \ref{prop: f(k,m) has a linear term}, a nonempty intersection $C_w \cap \cH(X,h)$ is cut out in $C_w$ by equations in Schubert variables with no constant terms. Hence such any nonempty intersection contains the flag $V_0$ in $C_w$, in which all of the Schubert variables take value zero. Note that $V_0$ is the permutation matrix associated to $w$. Hence we may check the behavior of pivots of $V_0$ under the action of the highest form operator $X$.
    
    The $k$-th column $v_k$ of $V_0$ is $e_{w(k)}$. Because $X$ is in highest form, $Xv_k = X e_{w(k)}  = e_{\phi(w(k))}$ if $w(k) > \dim(\ker(X)))$, and $Xv_k =0$ otherwise. Hence for indices $k$ such that $w(k) \leq \dim(\ker(X))$, we must take $h_{\operatorname{min}}(k) = \min\{k, h_{\operatorname{min}}(k-1)\}$, the minimal value at $k$ that a Hessenberg function is allowed to take. If $w(k) > \dim(\ker(X))$, then the subspace $V_l$ contains the image vector $Xv_k =e_{\phi(w(k))}$ if and only if $l \geq w^{-1}(\phi(w(k)))$. Hence in this case we take $h_{\min}(k) = w^{-1}(\phi(w(k)))$.
    Therefore, we obtain a recursive formula for $h_{\min}$ as follows.
    \begin{align*}
        h_{\min}(k) = \begin{cases}
            \max\{k,  h_{\min}(k-1)\} &\text{if  } w(k) \leq \dim(\ker(X)), \\
            \max\{k, h_{\min}(k-1), w^{-1}(\phi(w(k)))\}  &\text{if  } w(k) > \dim(\ker(X)) .
        \end{cases}
    \end{align*}

\noindent By construction of $h_{\min}$, this minimal Hessenberg function is unique. Indeed, for any index $k$, if $h$ is a Hessenberg function such that $h(k) < h_{\min}(k)$, then $Xv_k$ does not lie in $V_{h(k)}$ from the argument above, so $C_w \cap \cH(X,h)$ is empty.  
\end{proof}

In practice, we determine $h_{\min}$ in an ascending order of the indices: we first determine $h_{\min}(1)$, which can then be used to determine $h_{\min}(2)$, and so on.

\begin{ex}
Consider $w=[4213]\in S_4$ and $\lambda=(2,2)$. Then the highest form $X$ has pivot row-indices $\phi(k) = (0,0,1,2)$  as shown in Example \ref{ex: highest form of (2,2)}. Following Theorem \ref{thm: Thm 2}, we compute $h_{\operatorname{min}}$ inductively in $k$ to obtain 
$h_{\text{min}}=\hleft 2234 \hright$. This is consistent with the diagram in Section 4.3 shown in Figure \ref{fig: (2,2), w=[4213]}.
\end{ex}

\begin{remark}
    We contrast this result with the original nonempty condition in Theorem \ref{thm: Cor 6.3} given by Tymoczko, which is stated as $w^{-1}Xw \in H(h)$. Observe that the operator $w^{-1}Xw$ sends the basis vectors $e_k$ to $e_{w^{-1}(\phi(w(k)))}$ if $w(k) > \dim(\ker(X))$  and 0 otherwise. The matrix $w^{-1}Xw$ to lie in $H(h)$ if and only if for each $k$, the image of $e_k$ lies in the span of $e_1,\cdots, e_{h(k)}$. This condition precisely amounts to the condition $h \geq h_{\operatorname{min}}$.
\end{remark}

%% file: refs.bib
@article{billey2025,
  title={Introduction to the Cohomology of the Flag Variety},
  author={Sara C. Billey and Yibo Gao and Brendan Pawlowski},
  journal={arXiv:2506.21064},
  year={2025},
  eprint={2506.21064},
  primaryclass={math.CO},
  archivePrefix={arXiv},
  url={https://arxiv.org/abs/2506.21064}
}

@book{Fulton1997,
  author    = {William Fulton},
  title     = {Young Tableaux: With Applications to Representation Theory and Geometry},
  series    = {London Mathematical Society Student Texts},
  volume    = {35},
  publisher = {Cambridge University Press},
  address   = {Cambridge},
  year      = {1997}
}

@Article{Tymoczko2006,
  author    = {Julianna S. Tymoczko},
  journal   = {American Journal of Mathematics},
  title     = {Linear conditions imposed on flag varieties},
  year      = {2006},
  issn      = {1080-6377},
  number    = {6},
  pages     = {1587-1604},
  volume    = {128},
  doi       = {10.1353/ajm.2006.0050},
  publisher = {Project MUSE},
}

@incollection{AbeHoriguchi2019,
  author    = {Abe, Hiraku and Horiguchi, Tatsuya},
  title     = {A Survey of Recent Developments on Hessenberg Varieties},
  booktitle = {Schubert Calculus and Its Applications in Combinatorics and Representation Theory},
  editor    = {Hu, Jianxun and Mihalcea, Leonardo C. and Li, Changzheng},
  series    = {Springer Proceedings in Mathematics \& Statistics},
  volume    = {332},
  pages     = {251--279},
  year      = {2020},
  publisher = {Springer Singapore},
  doi       = {10.1007/978-981-15-7451-1_10}
}

@article{DeMariProcesiShayman1992,
  author    = {F. De Mari and C. Procesi and M. A. Shayman},
  title     = {Hessenberg varieties},
  journal   = {Transactions of the American Mathematical Society},
  volume    = {332},
  number    = {2},
  year      = {1992},
  pages     = {529--534},
  doi       = {10.2307/2154242}
}

@Article{Precup2013,
  author    = {Martha Precup},
  journal   = {Selecta Mathematica},
  title     = {Affine pavings of Hessenberg varieties for semisimple groups},
  year      = {2013},
  issn      = {1022-1824},
  number    = {4},
  pages     = {903-922},
  volume    = {19},
  doi       = {10.1007/s00029-012-0109-z},
  publisher = {Springer Science and Business Media LLC},
}

@article{dewitt2012poset,
  title={Poset pinball, highest forms, and $(n-2,2)$ Springer varieties},
  author={Dewitt, Barry and Harada, Megumi},
  journal={The Electronic Journal of Combinatorics},
  volume={19},
  number={1},
  pages={P56},
  year={2012},
  eprint={1012.5265},
  archivePrefix={arXiv},
  primaryClass={math.AG}
}

@article{deConciniProcesi,
  author  = {De Concini, C. and Procesi, C.},
  title   = {Complete symmetric varieties},
  journal = {Inventiones Mathematicae},
  volume  = {81},
  number  = {2},
  pages   = {275--302},
  year    = {1985},
  doi     = {10.1007/BF01389054}
}
